\documentclass[10pt]{amsart}
\usepackage{amsmath}
\usepackage{amsfonts}
\usepackage{amssymb}
\usepackage{amsthm}
\usepackage{url}
\usepackage[all]{xy}

\usepackage{dsfont}
\usepackage{graphicx}
\usepackage{caption}
\usepackage{subcaption}
\usepackage{comment} 
\usepackage{stmaryrd}
\usepackage{hyperref}
\usepackage{todonotes}
\usepackage{color}
\usepackage{enumerate,tikz-cd}
\usepackage{fixltx2e}
\usepackage{MnSymbol}
\usepackage{comment}
\allowdisplaybreaks

\numberwithin{equation}{section}

\newtheorem{proposition}{Proposition}[section]
\newtheorem{lemma}[proposition]{Lemma}
\newtheorem{theorem}[proposition]{Theorem}

\newtheorem{corollary}[proposition]{Corollary}

\theoremstyle{definition}
\newtheorem{remark}[proposition]{Remark}
\newtheorem{definition}[proposition]{Definition}
\newtheorem{example}[proposition]{Example}

\DeclareMathOperator{\Bl}{Bl}

\DeclareMathOperator{\Ima}{Im}

\DeclareMathOperator{\ord}{ord}
\DeclareMathOperator{\supp}{supp}

\DeclareMathOperator{\Ver}{Vert}
\DeclareMathOperator{\Rees}{Rees}

\newcommand{\R}{\mathbb{R}}
\newcommand{\C}{\mathbb{C}}

\newcommand{\Q}{\mathbb{Q}}

\newcommand{\pr}{\mathbb{P}}
\renewcommand{\epsilon}{\varepsilon}

\newcommand{\scO}{\mathcal{O}}

\newcommand{\mfa}{\mathfrak{a}}

\newcommand{\mfb}{\mathfrak{b}}

\newcommand{\E}{\mathcal{E}}
\renewcommand{\L}{\mathcal{L}}
\newcommand{\X}{\mathcal{X}}
\newcommand{\Y}{\mathcal{Y}}

\newcommand{\scZ}{\mathcal{Z}}

\renewcommand{\phi}{\varphi}

\newcommand\MA{\mathrm{MA}}

\newcommand\mi{^{-1}}
\newcommand\vol{\mathrm{vol}}
\newcommand\na{{\mathrm{NA}}}

\newcommand\bbc{\mathbb{C}}

\newcommand\bbp{\mathbb{P}}
\newcommand\bbq{\mathbb{Q}}
\newcommand\bbr{\mathbb{R}}

\newcommand\cE{\mathcal{E}}
\newcommand\cF{\mathcal{F}}
\newcommand\cG{\mathcal{G}}

\newcommand\cL{\mathcal{L}}
\newcommand\cM{\mathcal{M}}

\newcommand\cO{\mathcal{O}}

\newcommand\cX{\mathcal{X}}
\newcommand\cY{\mathcal{Y}}

\newcommand{\norm}[1][\cdot]{\left\|#1\right\|}

\newcommand\triv{{\mathrm{triv}}}
\newcommand\Bs{\mathrm{Bs}}

\newcommand\NA{^{\mathrm{NA}}}

\newcommand\red{{\mathrm{red}}}

\renewcommand\Big{{\mathrm{Big}}}

\renewcommand\div{\mathrm{div}}

\title[A birational version of K-stability for big classes]{A birational version of K-stability for big classes}
\author[Ruadha\'i Dervan and R\'emi Reboulet]{Ruadha\'i Dervan and R\'emi Reboulet}

\address{Ruadha\'i Dervan, Warwick Mathematics Institute, Zeeman Building
University of Warwick, Coventry CV4 7AL United Kingdom}\email{ruadhai.dervan@warwick.ac.uk}
\address{R\'emi Reboulet, CNRS, Institut Camille Jordan, 21 Av. Claude Bernard, 69100 Villeurbanne, France}\email{reboulet@math.univ-lyon1.fr}

\begin{document}

\maketitle

\begin{abstract}
We introduce a theory of uniform K-stability for big line bundles on smooth projective varieties. This extends the existing theory both  for varieties with ample line bundles, and for  varieties with big anticanonical class. Our main result gives a valuative characterisation of uniform K-stability, through finite collections of divisorial valuations. We further prove that uniform K-stability is preserved under pullbacks and certain pushforwards, which implies that uniform K-stability is well-defined at the level of $b$-divisors.
\end{abstract}

\tableofcontents

\section{Introduction}

The Yau--Tian--Donaldson conjecture, which states that the existence of a constant scalar curvature K\"ahler (cscK) metric on a smooth polarised variety can be characterised  algebro-geometrically, has been a major driving force of research in complex geometry over many decades. A solution to this conjecture has recently been provided by Boucksom--Jonsson \cite{bj:ytd} and Darvas--Zhang \cite{dar:zhangytd}, through the notion of \textit{uniform K-stability}\footnote{To simplify notation, we use the terminology of \emph{uniform K-stability} to refer to what is called \emph{uniform K-stability with respect to models} in the literature.}, building on  \cite{chencheng:iiiexistence,chili:ytd} and  extending the solution of Chen--Donaldson--Sun \cite{ytd:cds} for K\"ahler--Einstein metrics in the  Fano setting.

Beyond the connection with the existence of K\"ahler--Einstein metrics, the  theory of K-stability of Fano varieties has been successful due to its connection with the moduli theory of Fano varieties, and because K-stability can be explicitly verified for large classes of Fano varieties \cite{kmod:final,book:xukstab,book:threefolds}. In both of these applications, the main advantage in the Fano setting is that finite-generation properties hold (for filtrations, or section rings) with greater frequency than the more general setting. Furthermore, when finite-generation fails, there are typically ways of circumventing this by directly handling non-finitely generated objects.

One of the main points of the solution of the Yau--Tian--Donaldson conjecture was to use a stronger notion of stability, which is able to overcome this kind of lack of finite-generation. Said another way, while the Yau--Tian--Donaldson conjecture seeks a cscK metric in the first Chern class of an \emph{ample} line bundle, its solution nevertheless involves the theory of \textit{big} line bundles---whose section ring may not be finitely-generated---in a crucial way. 

The goal of this article is to develop a theory of uniform K-stability in a general setting in which finite-generation is not required for a working theory. More precisely, we consider a smooth projective variety $X$ along with a big line bundle $L$, and define a notion of \emph{uniform K-stability} of $(X,L)$. While ample line bundles capture the full geometry of a projective variety, the more flexible theory of big line bundles can be viewed as capturing the birational geometry of a projective variety; recall that a line bundle is said to be big if its volume (measuring the asymptotic growth of global sections of its tensor powers) is strictly positive, and that the section ring of such an $(X,L)$ may not be finitely-generated. In this way, our work can be viewed as producing a birational theory of uniform K-stability.

When $L=-K_X$ is the anticanonical class, as an analogue of the Fano setting, the authors and Darvas--Zhang previously introduced a notion of K-stability when $L=-K_X$ is big  \cite{derreb:1,dar:zhangbigke} (see also Trusiani \cite{trusiani:ytd} for related work). Via minimal model programme techniques, it was proven  in this case by Xu \cite{xu:big} that K-stability forces the anticanonical ring to be finitely generated, and that K-stability reduces to log K-stability of an associated log Fano pair, hence reducing the theory to the classical case. 

Our motivation to introduce a theory of uniform K-stability in the big setting is threefold. Firstly, as already explained, much of the difficulty behind the theory of K-stability in the ample case relates to the lack of finite-generation in various settings, and the theory of big line bundles has proven to be a powerful tool to circumvent these issues. From these developments, our perspective is that big line bundles, rather than ample line bundles, can be viewed as the natural home of the theory of K-stability, where the lack of finite-generation is built into the theory from the beginning.

Secondly, in the ample setting, uniform K-stability is conjectured to give rise to a moduli theory of polarised varieties. We ask whether there is a moduli theory for (birational equivalence classes of) projective varieties along with big line bundles, through uniform K-stability. The main obstacle in the ample setting is, again, related to finite-generation, and any results overcoming these difficulties are likely to apply also to the big setting. 

Thirdly, we expect the theory of constant scalar curvature K\"ahler metrics to extend to the big setting, by analogy with  recent important developments surrounding such metrics on singular varieties \cite{trusianipan,han-liu,to-pan}. We further expect  our notion of uniform K-stability to be equivalent to the existence of such metrics, suitably defined, extending the Yau--Tian--Donaldson conjecture beyond the ample cone.

\subsection*{Main results} We fix a smooth projective variety $X$ along with a big line bundle $L$. Our first results around uniform K-stability of $(X,L)$ are motivated by results in the ample case, which we briefly recall. 

When $L$ is ample, uniform K-stability is a numerical condition on $\C^*$-degenerations of $(X,L)$, which are called \emph{test configurations}. One then defines the \emph{non-Archimedean Mabuchi functional} (on the space of test configurations) by intersection theory, with \emph{uniform K-stability} asking for a uniform lower bound on the non-Archimedean Mabuchi functional, in terms of the norm of the test configuration. It was realised by Chi Li that it is beneficial to allow test configurations $(\X,\L)$, where the line bundle $\L$ is merely big \cite{chili:fujita,chili:ytd}, rather than ample, with intersection numbers being replaced by positive intersection numbers \cite{bfj:teissier}. This produces a version of K-stability which is called \emph{uniform K-stability with respect to models}, and which we abbreviate to simply uniform K-stability in our work (it being the notion that naturally extends to the big setting). A main point of Li's work is that uniform K-stability then recovers the \emph{a priori} stronger non-Archimedean approach of Boucksom--Jonsson \cite{bj:kstab1,bj:kstab2, bj:trivval}, and in particular is the notion of stability involved in the solution of the Yau--Tian--Donaldson conjecture.

In the Fano setting, the majority of the theory is reliant on a characterisation of K-stability through divisorial valuations, as opposed to test configurations. This characterisation, due to Fujita and Li \cite{fujita:valuativecriterion, chili:kstabequivariant}, employs powerful results from the minimal model programme due to Li--Xu \cite{lixu}, and has been fundamental to the purely algebro-geometric development of K-stability \cite{kmod:final,book:xukstab,book:threefolds}. 

In the general polarised case, this valuative picture was extended  by the first author and Legendre \cite{dervanlegendre}, Boucksom--Jonsson \cite{bj:kstab2,bj:ytd}, and Mesquita-Piccione--Witt Nystr\"om \cite{wn:pic}, through what is now known as \textit{divisorial stability}, where one tests stability through convex combinations of divisorial valuations on $X$ (which can be encapsulated into so-called \textit{divisorial measures}). In particular, Boucksom--Jonsson prove that divisorial stability is equivalent to uniform K-stability, generalising the aforementioned work of  Fujita and Li beyond the Fano setting. 

Turning to the case where $L$ is big, our theory of uniform K-stability is modelled on the work of Li: we consider $\C^*$-degenerations $\X$ of $X$ with $\L$ a big line bundle on $X$, which we call test configurations. To such a test configuration we associate numerical invariants, namely the \textit{non-Archimedean Mabuchi functional} and the \textit{norm}, defined by positive intersection theory. Uniform K-stability then requires a uniform lower bound on the non-Archimedean Mabuchi functional with respect to the norm, for all test configurations.

We further define a notion of divisorial stability, where one tests stability through convex combinations of divisorial valuations on $X$ (which can be encapsulated into so-called \textit{divisorial measures}). Our  main result proves the following characterisation in our setting, generalising the work of Boucksom--Jonsson\footnote{We emphasise that their work implies that divisorial stability is equivalent to uniform K-stability with respect to models, which we simply call uniform K-stability in our work.}.

\begin{theorem}\label{thm:intro-divisorial}
Suppose $X$ is a smooth projective variety, and $L$ is a big line bundle on $X$. Then $(X,L)$ is  uniformly K-stable if and only if it is divisorially stable.\end{theorem}

Many new technical issues arise in proving this result, in comparison with the ample case. The first step towards proving this result is to establish a non-Archimedean analogue of the Calabi--Yau theorem for big line bundles, as a big counterpart to ample results of Boucksom, Favre and Jonsson \cite{bfj:nama, bj:trivval}; we adapt a recent approach of Witt Nystr\"om \cite{wn:icm}, for which it is crucial that we assume $X$ is smooth. A major difference in our setting, however, is that we do not establish (and do not expect) uniqueness of solutions, due to a lack of strict convexity of the key Monge--Amp\`ere energy functional. This functional is defined through the volume of a big line bundle, for which one can only expect strict convexity properties on the big and nef (or movable) cone. Geometrically, this means we must permit a single divisorial measure to be associated with many test configurations.

The main consequence of this lack of uniqueness relates to the numerical invariants we attach to a divisorial measure. In the ample case, one essentially defines these invariants precisely so that the norm of a divisorial measure agrees with the norm of its associated test configuration, while the $\beta$-invariant of a divisorial measure is defined in such a way that it agrees with the non-Archimedean Mabuchi functional of its associated test configuration. It is therefore remarkable that it is possible to produce an intrinsic definition of the $\beta$-invariant in such a way that Theorem \ref{thm:intro-divisorial} holds. For this, we prove \emph{left} differentiability of the norm of a fixed measure as one varies the big line bundle, using convex optimisation results of Danskin \cite{danskin}; we do not expect full differentiability. 

We emphasise that our resulting notion of the $\beta$-invariant of a divisorial measure is completely explicit, involving the calculation only of integrated volumes of various collections of divisorial valuations, along with the log discrepancy, in line with the approach of \cite{bj:ytd, wn:pic} rather than the original approach of \cite{bj:kstab2}.

Another important aspect of our proof is to directly relate the Monge--Amp\`ere measure of a test configuration, the volume of an associated filtration of the section ring, and  the ``integrated volume'' of the associated divisorial measure, through a result of independent interest generalising equidistribution results in \cite{bouchen,bj:kstab1} to globally big test configurations; see  Theorem \ref{thm:bchen}. We use Theorem \ref{thm:intro-divisorial} to prove that uniform K-stability of $(X,-K_X)$ recovers the prior notions introduced in \cite{derreb:1,dar:zhangbigke}, implying that our notion of uniform K-stability is a natural extension of the existing theory in the anticanonical case.
 
Our second main result establishes the birational behaviour of uniform K-stability, for which we consider a birational contraction between smooth projective varieties $\pi:Y\dashrightarrow X$. We endow $X$ and $Y$ with big line bundles $L_X$ and $L_Y$ respectively such that on the pullback to a resolution of indeterminacy $Z$, we have $L_Y = L_X+E$, for $E$ an exceptional effective $\bbr$-Cartier divisor. 

\begin{theorem}\label{intro:birational}
$(X,L_X)$ is uniformly K-stable if and only if $(Y,L_Y)$ is uniformly K-stable. 
\end{theorem}

A special case to which this result applies is when $L_Y$ is the pullback of $L_X$ to $X$ through a morphism. Perhaps the most important conceptual consequence is in the case where $L_Y$ has finitely-generated section ring: in this case, one obtains a birational contraction $\pi: Y \dashrightarrow X$ and an ample line bundle $L_X$ on $X$ such that, on a resolution of indeterminacy, $L_Y = \pi^*L_X + E$ for $E$ effective. Thus, when $L_Y$ has finitely generated section ring---and its ample model $(X,L_X)$ satisfies the hypothesis that $X$ is smooth---uniform K-stability reduces to the ample setting. This is compatible with the minimal model programme and surrounding moduli theory: when $K_X$ is big, taking the canonical model produces a variety with ample canonical class, and Theorem \ref{intro:birational} implies that uniform K-stability is unchanged under passage to smooth canonical models (and indeed, in this case proving uniform K-stability of $(X,K_X)$ directly is straightforward). In this way, Theorem \ref{intro:birational} implies that we have produced the natural extension of uniform K-stability from the ample cone to the big cone. 

In line with the third point in our motivation above, Theorem \ref{intro:birational} is an algebro-geometric version of the statement that the existence of (a conjectural notion of) singular cscK metrics in big classes is well-behaved under pullbacks and pushforwards. As such, singular cscK metrics in big classes should be abundant, beginning for example  with a smooth cscK manifold $(X,L_X)$ with $L_X$ ample.

 Theorem \ref{intro:birational} further implies that uniform K-stability is a notion defined at the level of \textit{$b$-divisors} on the Riemann--Zariski space of $X$, as we explain in Section \ref{sec:b-div}. Many fundamental notions of positivity for classes are defined in terms of classes on the Riemann--Zariski space \cite{bfj:teissier,dangfavre:nefbdiv}, which provides the appropriate framework for (asymptotic) birational-geometric questions, and our work thus gives a theory of uniform K-stability that fits into this framework.

We end with a more philosophical comment on our approach to uniform K-stability in the big setting, which is motivated by considerations from non-Archimedean geometry. Boucksom--Jonsson \cite{bj:trivval} have given a general theory of plurisubharmonic  metrics in non-Archimedean geometry, for arbitrary pseudoeffective classes. Our work is strongly influenced by theirs, but does not use the same approach; our notion of a test configuration is distinct from their notion of a non-Archimedean Fubini--Study metric. 

More precisely, in the complex (Archimedean) theory,  a plurisubharmonic metric can be viewed as a decreasing limit of (in general singular) Fubini--Study metrics through Demailly approximation \cite{dem:reg}, which is taken as a definition in non-Archimedean pluripotential theory. Our approach  is instead motivated by the fact that plurisubharmonic envelopes of continuous functions are dense in the space of metrics with minimal singularities; our test configurations are analogues of such plurisubharmonic envelopes. In the Boucksom--Jonsson theory, a major issue is that the space of finite energy plurisubharmonic non-Archimedean metrics with minimal singularities is not known to be complete in general, even in the big and nef case (see \cite{darxiazhang} for related results using the formalism of \cite{rwn:analytictest}). The approach which we take means we do not require analogous completeness results to obtain a working theory, such as in our approach to the non-Archimedean Monge--Amp\`ere equation.

\bigskip\noindent\textbf{Acknowledgements.} The authors thank Sébastien Boucksom, Tamás Darvas, Jiyuan Han, Mattias Jonsson, Eveline Legendre, Yaxiong Liu,  Lars Martin Sektnan, Pietro Mesquita Piccione, David Witt Nyström, and Kewei Zhang for discussions related to this article and comments on an earlier version. RD was funded by a Royal Society University Research Fellowship (URF/R1/20104) for the duration of this work.

\section{K-stability for big line bundles} \label{sect:1}

We will be interested throughout in  a smooth complex projective variety $X$ of dimension $n$, with canonical class $K_X$. For a big line bundle $L$ on $X$, our aim is to define the pair of conditions that $(X,L)$ be uniformly K-stable and divisorially stable. In subsequent sections, we will prove the equivalence of the two conditions. 

\subsection{Preliminaries on positivity and volume}

We begin by recalling the basics of positivity theory in algebraic geometry, referring to Lazarsfeld and Boucksom--Favre--Jonsson for introductions \cite{book:laz1, book:laz2, bfj:teissier}. In the present section, we assume that $X$ is merely normal.
\bigskip

\subsubsection{Positivity of line bundles.} We recall that a line bundle $L$ on $X$ is \textit{big} if its volume
$$\vol(L):=\lim_{k\to\infty}\frac{h^0(X,kL)}{k^n/n!}$$
is strictly positive; it is a classical result that the limit involved in the definition exists. Equivalently, $L$ is big if and only if $L=A+E$ with $A$ ample and $E$ effective on $X$. Bigness can be generalised in the natural way to $\bbr$-divisors and is a numerical property. In particular, the \textit{big cone} in $N^1(X)$ is open \cite[Corollary 2.2.24]{book:laz1}. A line bundle $L$ is \textit{pseudoeffective} if its numerical class lies in the closure of the big cone. We say that $L\geq L'$ if $L-L'$ is pseudoeffective.

Given a line bundle $L$ on $X$, we define its \textit{stable base locus} $Bs(||L||)=\bigcap_{k>0} Bs|kL|;$ its \textit{augmented base locus} or \textit{non-nef locus} $B_+(L):=\bigcap_A Bs(||L-A||)$, with $A$ varying over all ample $\bbq$-divisors on $X$; and its \textit{restricted base locus} or \textit{non-ample locus} $B_-(L):=\bigcup_A Bs(||L+A||)$ 	with $A$ varying over all ample $\bbq$-divisors on $X$.

We have $B_-(L)\subseteq B(L)\subseteq B_+(L)$; furthermore, $L$ is pseudoeffective if and only if $B_-(L)\neq X$, and $L$ is big if and only if $B_+(L)\neq X$. Given a divisor $D$, we say a pseudoeffective line bundle $L$ is \textit{$D$-pseudoeffective} if $D$ is not contained in $B_-(L)$, and \textit{$D$-big} if $D$ is not contained in $B_+(L)$ \cite[Remark 4.2]{bfj:teissier}. Note that the $D$-pseudoeffective cone is the closure of the $D$-big cone, which conversely is the interior of the former.

Given a big line bundle $L$ and a divisor $E$ on $X$, we define the \textit{pseudoeffective threshold} of $L$ with respect to $E$ as
$$\gamma(L,E):=\sup\{\gamma\in\mathbb{R}_{>0},\,L-\gamma E\text{ is big}\}.$$

\subsubsection{Positive intersection theory.} For $L$ a big line bundle on $X$, one defines the \textit{positive intersection product} as in \cite[Definition 2.5]{bfj:teissier}:
$$\langle L^k\rangle:=\sup_D\, (\pi^*L-D)^k,$$
where the supremum (taken in the sense of the pseudoeffectivity partial order $\geq$ as above) ranges over all divisors $D$ in a birational model $\pi:Y\to X$ such that $\pi^*L-D$ is nef. Formally, $\langle L^k\rangle$ is viewed as an element of $N^k(\mathfrak X)$, the space of $p$-dimensional Weil classes on the Riemann--Zariski space $\mathfrak X$ of $X$ \cite[Definition 1.1]{bfj:teissier}, though we require very little of the general theory (though we refer to Section \ref{sec:b-div} for a more complete discussion).  We note $\langle L^k\rangle = \langle L\rangle^k$ in the resulting positive intersection theory, and we employ both notations throughout. 

The positive intersection product extends to the boundary of the pseudoeffective cone in a natural manner in \cite[Definition 2.10]{bfj:teissier}, and to zero outside the pseudoeffective cone. It is not a continuous extension in general, except in the case where $k=n$. 

Provided $L$ is big, its volume is computed as a positive intersection product:
\begin{equation}
\vol(L)=\langle L^n\rangle.
\end{equation}

\subsubsection{Restricted volumes and differentiability of volume} The \textit{restricted volume} of a line bundle $L$ on a subvariety $V\subset X$ is defined as
$$\vol_{X|V}(L):=\limsup_{k\to\infty}\frac{\dim \Ima(H^0(X,kL)\to H^0(V,kL|_{V}))}{k^{\dim V}/\dim V!}.$$
If $V=D$ is a prime divisor in $X$, and $L$ is big, then \cite[Theorem B]{bfj:teissier}
\begin{equation}
\vol_{X|D}(L)=\langle L^{n-1}\rangle\cdot D.
\end{equation}

In general, for an arbitrary line bundle $L$, one has the following inequalities:
\begin{equation}\label{eq:restrineq}
	0\leq \vol_{X|D}(L)\leq \langle L^{n-1}\rangle\cdot D\leq \langle L^{n-1}\rangle|_D\leq \langle L|_D^{n-1}\rangle,
\end{equation}
the last inequality following from \cite[Definition 4.4, Remark 4.5]{bfj:teissier}. By \cite[Theorems 4.14, 4.15]{bfj:teissier}, if $L$ is $D$-big, or big and $D$-pseudoeffective, then further
$$\vol_{X|D}(L)=\langle L^{n-1}\rangle|_{D}.$$

The volume function is known to be $C^1$ on the big cone (\cite[Theorem A]{bfj:teissier}, \cite{lazmus}): given $L$ big and any divisor $D$ on $X$,
\begin{equation}
	\frac{d}{dt}\bigg |_{t=0} \vol(L+tD)=n\langle L^{n-1}\rangle\cdot D.
\end{equation}
We note that it is not differentiable at the boundary of the pseudoeffective cone.

We will use the following property of the positive intersection product repeatedly. Let $\pi: Y \to X$ be a birational morphism between normal varieties, and let $E$ be an irreducible divisor on $X$ with proper transform $\pi_*^{-1}E$ on $Y$. Then for any big line bundle $L$ on $X$ \cite[Section 4.2]{bfj:teissier} \begin{equation}\label{tess-1}\langle L^{n-1}\rangle\cdot E = \langle \pi^*L^{n-1}\rangle \cdot \pi_*^{-1}E,\end{equation} and similarly for $F$ a $\pi$-exceptional divisor on $Y$, \begin{equation}\label{tess-2}\langle \pi^*L^{n-1}\rangle \cdot F = 0.\end{equation}

\subsubsection{Fujita approximation.} The positive intersection products $\langle L^n\rangle$ and $\langle L^{n-1}\rangle\cdot  D$ can be computed using an explicit limit, following \cite[Definition 2.6, Theorem 2.13]{elmnp:restricted}. Let $\mu_m:X_m\to X$ be the blowup of the base locus of $mL$, then we may choose $m\mi E_m$ to be a maximising sequence in the definitions of the positive products above, so that
\begin{equation}\label{eq:fujita}
\lim_{m\to\infty}(m\mi L_m)^n=\langle L^n\rangle,\quad \lim_{m\to\infty}(m\mi L_m)^{n-1}\cdot D=\langle L^{n-1}\rangle\cdot D.
\end{equation}
We will throughout refer to the data of $(X_m,L_m:=\pi^*(mL)-E_m)$ as \textit{the Fujita approximation} of $(X,L)$

\subsection{K-stability}\label{subsect:tc}

We fix a smooth projective variety $X$ along with a big line bundle $L$ on $X$, and next define uniform K-stability of $(X,L)$. While the results of our work  hold only when $X$ is smooth, the definitions in the present section make sense provided $X$ is normal.

\begin{definition}
Consider a normal projective variety $\X$ with a $\C^*$-action, a flat $\C^*$-equivariant morphism $\pi: \X \to \pr^1$ and a $\C^*$-equivariant isomorphism $\X - \X_0 \cong X \times \C$, and consider $\L$  an $\R$-line bundle on $\X$ such that $\L|_{\X - \X_0} \cong L$. We say that $(\X,\L)$ is a \emph{test configuration} for $(X,L)$, and that $\X$ is a \emph{degeneration} of $X$.\end{definition}

\begin{example} If $\X = X \times \pr^1$ with $\L = p_X^*L$ the pullback to $\X$, we say $(\X,\L)$ is the \emph{trivial} test configuration.
\end{example}

An important class of test configurations are those that are \emph{dominant}, which means that the natural $\C^*$-equivariant rational map $\psi: \X \dashrightarrow X \times \pr^1$ extends to a morphism $\X \to X \times \pr^1$. For an arbitrary test configuration, by passing to a $\C^*$-equivariant resolution of indeterminacy $q: \Y \to \X$ of $\psi$, one obtains a new test configuration $(\Y,q^*\L)$.  For a dominant test configuration $(\X,\L)$, we will omit pullbacks of line bundles to $\X$ from $X$ and $\pr^1$ in our notation.

\begin{remark}\label{rmk:big}Given a test configuration $(\X,\L)$, we  show $\L+\scO_{\pr^1}(k)$ is big for any sufficiently large $k>0$. As the property of being big is a birational invariant, we may assume $(\X,\L)$ is dominant and that $\X \to X\times\pr^1$ admits a relatively ample, exceptional line bundle $E$, so that  we have $\L = L + D$ for a unique vertical Cartier divisor $D$ on $\cX$. Note that, after adding $\scO_{\pr^1}(k)$ for some $k$,  we may assume $\L - L - \scO_{\pr^1}(1)$ is effective. Since $L$ is big, we may write $L=A+F$ for $A$ ample and $F$ effective, so that $$\L = (A+ \scO_{\pr^1}(1) - \epsilon E)+(\epsilon E+(\L-L) + F - \scO_{\pr^1}(1))$$ is the sum of an ample class on $\X$ (namely $A+ \scO_{\pr^1}(1) - \epsilon E$) and an effective class (namely $(\epsilon E+(\L-L)+F - \scO_{\pr^1}(1))$). Thus, up to a twist by  $\scO_{\pr^1}(k)$ with $k$  sufficiently large, any test configuration can be made so that $\cL$ is globally big on $\cX$, and the same holds for any $k'>k$. We will associate numerical invariants to test configurations $(\X,\L)$ by first twisting them in this way so that they become big, many of which will be independent of choice of $k$; for the others, we will determine their dependence on $k$.
\end{remark}

We turn to numerical invariants, for which we assume $(\X,\L)$ is dominant; the various invariants will be independent of choice of resolution of indeterminacy for general test configurations.

Consider a dominant test configuration $(\X,\L)$. We then obtain a natural divisor $D$ by setting $\L - L = D$ for $D$ supported in $\X_0$, so that $D$ is vertical. The following results essentially appear in work of Trusiani  \cite[Lemma 2.15]{trusiani:ytd}, and we provide a proof adapted to our setting.

\begin{proposition}\label{trusiani} 
Suppose $\L - L = D$ is effective.  Then:
\begin{enumerate}[(i)]
\item for all $m \geq 0$, the restriction of the base locus $\Bs|m\L|$ to $\X - \X_0$ is $\Bs|mL| \times \C$;
\item $\vol(\L + \scO_{\pr^1}(k)) = \vol(\L) + k(n+1)\vol(L).$
\end{enumerate}
\end{proposition}

\begin{proof}

To prove the first claim, suppose $x \notin \Bs(m\L)$ for $x \in \X_t$ with $t\neq 0$. Then there exists a $s \in H^0(\X,m\L)$ with $s(x)\neq 0$; since $m\L|_{\X_t}\cong mL$, so restriction produces a section $s'\in H^0(X,mL)$ with $s'(x) \neq 0$, and so $x \notin \Bs|mL| \times \C$. Here we use that the pullback of the base locus of $mL$ under the morphism $X \times \pr^1 - X\times\{0\} \to X$ is simply the preimage of the base locus of $mL$.

Conversely, for $t\neq 0$ if $x\in \X_t$ is not contained in $\Bs|mL|$ (more precisely, $(\Bs|mL|)|_{\X_t}$), then we may choose a section $s\in H^0(X,mL)$ with $s(x) \neq 0$. This section extends to a section $s'$ of the pullback of $mL$ to $\X$, and since $m\L - mL = mD = V(s_D)$ is effective by hypothesis, we obtain a section of $H^0(\X,m\L)$ by considering $s'\otimes s_D$ which satisfies $ (s'\otimes s_D)(x) \neq 0$ as $D$ is vertical. Thus $x$ is not contained in the restriction of the base locus of $m\L$ to $\X - \X_0$, which implies the result.

To prove the second claim, we claim $(i)$ implies $\vol_{\X|X}(\L)$ agrees with the volume of $L$. Indeed, in general, $\vol_{\X|X}(\L) \leq \vol(\L|_{\X_t}) = \vol(L)$, while the previous argument produces, for any section $s\in H^0(X,mL)$ a section $s' \in H^0(\X,m\L)$ whose restriction is $s$, meaning $\vol(L) \leq \vol_{\X|X}(\L).$ This is enough to imply the result by differentiability of the volume, since \begin{align*}\frac{d}{dt}\big|_{t=0} \vol(\L + (1+t)\scO_{\pr^1}(k)) &= (n+1)\vol_{\X|X}(\L), \\ &= (n+1)\vol(L);\end{align*} this implies the function of $t$  given by $t \to \vol(\L + (1+t)\scO_{\pr^1}(k))$ is affine of slope $(n+1)\vol(L)$, and equals $\vol(\L)$ at $t=0$, which gives the result.\end{proof}

By replacing $\L$ with $\L+\scO_{\pr^1}(k)$, we can ensure $\L - L$ is effective. This result allows us to make the following definition.

\begin{definition}

The \emph{Monge--Amp\`ere energy} of $(\X,\L)$ is given by $$E\NA(\X,\L) = \frac{\langle\L + \scO_{\pr^1}(k)\rangle^{n+1}}{(n+1)\langle L^n\rangle} - k$$ for $k \gg 0$.   

\end{definition}

Let $(\X,\L)$ be a test configuration, and assume that $\X$ is normal. As $\X$ is normal, we may  make sense of its canonical class $K_{\X}$ as a Weil divisor, and we define $$K^{\log}_{\X} = K_{\X} + (\X_0 - \X_{0,\red}),$$ where $\X_{0,\red}\subset \X_0$ is the reduced subscheme. We denote the canonical class of $X$ by $K_X$, which is a $\Q$-line bundle as $X$ is assumed $\Q$-Gorenstein.

\begin{definition} Associated to a test configuration $(\X,\L)$ with  normal total space, we define:
\begin{enumerate}[(i)]
\item its \emph{entropy} by $$H\NA(\X,\L) = \vol(L)\mi\left(\langle \L^n\rangle \cdot K^{\log}_{\X/X\times\pr^1}\right);$$
\item the \emph{non-Archimedean Mabuchi functional} by $$M\NA(\X,\L) = H\NA(\X,\L) + \nabla_{K_X} E\NA(\X,\L);$$
\item its \emph{norm} by $$\|(\X,\L)\| = \nabla_L E\NA(\X,\L).$$
\end{enumerate}

\end{definition}

To elaborate upon the definition of the entropy, $\langle \L^n\rangle \cdot K^{\log}_{\X/X\times\pr^1}$ denotes the positive intersection product of $\langle \L^n\rangle$ with the class $K_{\X} + (\X_0 - \X_{0,\red}) - K_{X\times\pr^1}.$ The directional derivatives involved in the definitions are given by, for example, $$\nabla_{K_X} E\NA(\X,\L) = \frac{d}{dt}\bigg |_{t=0}E\NA(\L+\scO_{\pr^1}(k)+tK_X),$$ and can be explicitly written as positive intersection products, by differentiability of the volume.

We may now define the first main stability condition of interest to us.

\begin{definition} We say that $(X,L)$ is \emph{uniformly K-stable} if there exists an $\epsilon>0$ such that for all test configurations $(\X,L)$ we have $$M\NA(\X,\L) \geq \epsilon \|(\X,\L)\|.$$ We call the maximal such $\epsilon$ the \emph{modulus of stability} of $(X,L)$.\end{definition}

\begin{remark}
Suppose $L$ is ample. Then, what we call uniform K-stability is usually called ``uniform K-stability with respect to models'' \cite{chili:ytd,chili:fujita,bj:kstab2,bj:ytd}, where a \textit{model} is a test configuration as we have defined one \cite{chili:ytd,chili:fujita}, while uniform signifies the uniform lower bound on $M\NA(\X,\L)$ in terms of the norm \cite{bhj:duistermaat,RD-uniform}.
\end{remark}

We note the following properties of uniform K-stability, which are standard in the ample setting (see e.g.\ \cite{ADVLN} or \cite[Proposition 7.14]{bhj:duistermaat}). 

\begin{lemma}\label{lem:can-assume-snc} To test uniform K-stability, we may restrict to test configurations with reduced, simple normal crossings central fibre and smooth total space.  \end{lemma}

\begin{proof}
By semistable reduction, associated to a test configuration $(\X,\L)$ is a new test configuration $(\Y,\pi^*\L)$ where $\pi: \Y \to \X$ is a generically finite, birational morphism, such that $\Y_0$ is reduced and simple normal crossings, covering a finite morphism $\pr^1\to \pr^1$ branched only at the origin. As Weil divisors, we may write $K_{\Y}^{\log} = \pi^*K_{\X}^{\log} + \E,$ where $\E$ is $\pi$-exceptional.  By Equation \eqref{tess-2} it follows that $\langle\pi^*\L\rangle\cdot\E=0$, while $\langle \pi^*\L \rangle^n\cdot \pi^*K_{\X}^{\log} = \langle \L\rangle^n\cdot K_{\X}^{\log}$ by properties of the log canonical class under base change (see \cite[Proposition 7.14]{bhj:duistermaat}, \cite[Section 3]{lixu}), and for example $\langle \pi^*\L\rangle^{n+1} = \langle \L\rangle^{n+1}.$ It is then straightforward to see $$M\NA(\X,\L) = M\NA(\Y,\pi^*\L).$$ Since $X$ is assumed smooth, we may further assume $\Y$ is a resolution of singularities of $\X$, and the same argument applies. 
\end{proof}

\begin{remark}
As the positive intersection product is numerical, and the collection of test objects is independent of $L$ (simply being degenerations along with vertical $\R$-Cartier divisors), it follows that uniform K-stability is a numerical condition.
\end{remark}

\subsection{Divisorial stability}\label{subsect:divstab}

We consider a finite collection of divisors $F_i$ which are rational positive multiples of a prime divisor over $X$, which means that for each $i$, $F_i=c_i\cdot E_i$ with $c_i\in\bbq_{>0}$, and $E_i$ is realised as a prime divisor on a birational model $Y_i \to X$ of $X$; we will typically view the $F_i$  through their associated divisorial valuations. We also conventionally set $F_\triv$, which corresponds to the \textit{trivial valuation} $\nu_\triv$ on the function field of $X$.

\begin{definition} A \emph{divisorial measure} is a collection $ \mu=\{(F_i,\xi_i)\}$ for $i\in I = \{0,\ldots, l\}$ a finite set, with $\xi_i\in [0,1]$ satisfying the normalisation condition $\sum_{i=0}^l \xi_i = 1.$ We call $\Sigma = \{F_i\}_{i\in I}$ the \textit{support} $\supp(\mu)$ of $\mu$. \end{definition}

We consider divisorial measures to be the ``test objects'', and wish to associate numerical invariants to a given divisorial measure. The terminology behind divisorial measures will be justified in Remark \ref{rmk:why-na}; to define divisorial stability, it suffices for now to view them as weighted collections of divisorial valuations.

The relevant numerical invariants will be defined through a Legendre transform, so we consider an additional collection of real numbers $(t_i)_i\in\R^{|I|}$.

\begin{definition}\label{def:SL} Let $\Sigma=(F_i)_{i\in |I|}$ be a finite collection of divisors over $X$ as above. We define the \emph{expected vanishing order} of $L$ along $(\Sigma,t)$ to be  $$S_{L,\Sigma}(t) = t_0 + \frac{1}{\vol(L)}\int^{\infty}_{t_0}\vol\left(X,L-\sum_{i=0}^l\max(\lambda-t_i,0)F_i\right)\,d\lambda,$$ where $t_0$ is by convention the smallest of the $t_i$. When the support $\Sigma$ is clear from context, we write this as $S_L(t)$ or $S_L(\{F_i,t_i\})$.
\end{definition}

This is well-defined independently of the choice of such $t_0$, as for $t<\min_i t_i$, the integrand is constant and equal to $\vol(L)$. Furthermore, the integrand is zero as soon as $t\geq \min_i\{\gamma(L,E_i)+t_i\}$, and $t_0$ can be replaced with any real number smaller than $t_0$ leaving the quantity unchanged.

\begin{definition}\label{def:norm-def}
Consider a divisorial measure $\mu=\{(F_i,\xi_i)\}$, and let $\Sigma$ be its support. We define its \textit{norm} to be the Legendre transform
$$\|\mu\|_L := \sup_{t\in \R^{|I|}}\{-\langle \xi, t\rangle + S_{L,\Sigma}(t)\}.$$
\end{definition}

Note that, for $c\in \bbr$,
$$-\langle \xi, t+c\rangle + S_{L,\Sigma}(t+c)=-\langle \xi,t\rangle + S_{L,\Sigma}(t),$$
using that $c\sum_i \xi_i=c$, a change of variables, and the fact that the integral $S_{L,\Sigma}(t+c)$ can be computed on $[s+c,\min_i\{\gamma(L,E_i)+t_i\}+c]$ for any $s\leq \min_i t_i$. We will thus often take the supremum over all $t$ normalised so that $\min_i t_i=0$, in which case the relevant integral takes the simpler form $$S_{L,\Sigma}(t) = \frac{1}{\vol(L)}\int^{\infty}_{0}\vol\left(X,L-\sum_{i=0}^l\max(\lambda-t_i,0)F_i\right)\,d\lambda,$$

\begin{remark}\label{rem:expectedorder}Suppose $\mu$ is divisorial with support a single valuation, hence of the form $\{(F,1)\}$. Then, from the normalisation discussion above, $\|\mu\|_L$ is genuinely computed at $t=0$, hence
	$$\|\mu\|_L=S_L(F)$$
	recovers the usual expected vanishing order \cite{fujita:valuativecriterion, chili:kstabequivariant}.
\end{remark}

We are interested in differentiating $U$ with respect to $L$, in the direction of $K_X$, for fixed $\xi$. Given another line bundle $M$ on $X$, we denote by $$\nabla^-_{M}\|\mu\|_L:=\lim_{t\to 0^-}\frac{\|\mu\|_{L+tM}-\|\mu\|_L}{t}$$ the (left) directional derivative in the direction of $M$. It is clear that $\nabla^-_M \|\mu\|_L=-\nabla^+_{-M}\|\mu\|_L$ with $\nabla^+$ the standard (i.e.\ right) directional derivative. We will prove existence of directional derivatives of $\|\mu\|_L$ in Theorem \ref{thm:diff}, which will use our assumption that $X$ is smooth. In particular, to compute $\nabla_{K_X}^{-} \|\mu\|_L$  requires only the calculation of $S_L$ over $t\in\R^{|I|}$ and (left) differentiation.

\begin{definition}\label{def:div-stab}
Suppose $X$ is smooth. For a divisorial measure $\mu=\{(F_i,\xi_i)\}$, we define its \emph{$\beta$-invariant} by  $$\beta_L(\mu) = \sum_i \xi_iA_X(F_i) + \nabla_{K_X}^- \|\mu\|_L.$$ We then say that $(X,L)$ is \emph{divisorially stable} if there exists an $\epsilon>0$ such that for all divisorial measures   $\mu=\{(F_i,\xi_i)\}$ we have $$\beta_L(\mu)  \geq \epsilon \|\mu\|_L.$$
\end{definition}

Here we denote by $A_X$ the log discrepancy function on $X$, which for $\pi:Y\to X$ a birational morphism with $Y$ normal, $\Q$-Gorenstein and $F\subset Y$ a prime divisor takes value $A_X(F) = \ord_F(K_Y-\pi^*K_X)+1$. As $X$ is smooth, $A_X(F)\geq 0$ for all divisorial valuations $F$ over $X$.

\begin{example} Suppose the divisorial measure is given by a single divisorial valuation. Then $$\beta(F) = A_X(F) + \nabla_{K_X} S_L(F),$$ which is the quantity introduced by the first author and Legendre \cite{dervanlegendre} in the case where $L$ is ample. In the case that further $L=-K_X$, this is the usual $\beta$-invariant involved in the theory of K-stability \cite{fujita:valuativecriterion, chili:kstabequivariant}, and was first considered by Darvas--Zhang \cite{dar:zhangbigke} when $-K_X$ is merely assumed big, rather than ample. We will show in Theorem \ref{thm:fano} that stability in the sense of \cite{dar:zhangbigke} is equivalent to divisorial stability in the sense of Definition \ref{def:div-stab} when $L=-K_X$ is big. 

In the case where $L$ is ample, the norm $\|\mu\|_L$ is genuinely differentiable in $L$, and this is the quantity introduced by Boucksom--Jonsson \cite{bj:kstab2}, where the more explicit expression we build on appears in \cite[p.\ 47]{chili:slides}, \cite[Section 10]{wn:pic}, \cite[Proposition 4.7]{bj:ytd}.
\end{example}

\section{A Calabi--Yau theorem}\label{sect:2}

We assume throughout this section that $X$ is smooth.

\subsection{The Monge--Amp\`ere measure of a test configuration}

To a test configuration $(\X,\L)$, we may associate a divisorial measure $\mu = \{E_i,\xi_i\}$ essentially by letting $E_i$ be the components of $\X_0$ and letting $\xi_i$ be a normalised version of $\langle \L^n\rangle\cdot E_i$. Our interest in the present section is  the inverse question: given a divisorial measure $\mu$, can we produce a test configuration whose associated measure is $\mu$? Concretely, can we prescribe the restricted volumes of a big line bundle along the central fibre of a test configuration? This question can be phrased through non-Archimedean geometry, where it corresponds to solving a non-Archimedean Calabi--Yau theorem.

We fix a degeneration $\X$ of $X$, and consider some general theory relating the components of $\X_0$ to divisorial valuations. Denote by $\X_0 = \sum_i b_iE_i$ the central fibre of $\X_0$, with the $E_i$ reduced and irreducible. The $E_i$ may be viewed as divisorial valuations on $X$ \cite[Theorem 4.6]{bhj:duistermaat}: one may assume that $\X$ is dominant, through the behaviour of divisorial valuations under birational morphisms. A rational function on $X$ thus induces one on $\X$ by pullback, and one simply takes the order of this rational function along the corresponding component of $\X$.

In the following statement and throughout, it will be useful to denote the set of \emph{divisorial valuations} on $X$ by $X^\div$ (where conventionally we include the trivial valuation as a divisorial valuation); this can be naturally given the structure of a topological space, but we have no need for this further structure. Writing $\X_0 = \sum b_i E_i$, we define the set of \textit{Rees valuations} of $\X$ to be the divisorial valuations $b_i\mi \ord_{E_i} \in X^\div$ on $X$. We make the following definition.

\begin{definition}
	We say a degeneration $\cX$ of $X$ is \textit{$\mu$-compatible} if $\supp(\mu)\subseteq \mathrm{Rees}(\cX)$ as subsets of  $X^\div$, with each $E_i$ associated to a $\Q$-Cartier divisor on $\X$. 
\end{definition}

As a starting point towards the Calabi--Yau theorem, given a divisorial measure $ \{E_i,\xi_i\}$, we note that we may at least find an $\X$ whose components of the central fibre contains the $E_i$, namely $\supp(\mu)\subseteq \mathrm{Rees}(\cX)$.

\begin{lemma}\label{lem:rees}\cite[Lemma 2.12]{bj:trivval}
	For any finite subset $S\subset X^\div$, there exists a dominant degeneration $\cX$ of $X$ such that $S$ is contained in the set of Rees valuations of $\cX$
\end{lemma}

This being a statement which does not involve $L$, it applies equally to the setting of \cite{bj:trivval} (namely ample line bundles) as to our own. It is clear that $\X$ can further be taken to be dominant, while since $\X$ can further taken to be smooth as $X$ is smooth, that we may assume that the $E_i$ are associated to a $\Q$-Cartier divisor is automatic. This is the principal manner in which smoothness is used in the present work.

We next involve the line bundle $\L$, by considering a test configuration $(\cX,\cL)$ with $\cX_0=\sum b_i E_i$.

\begin{lemma}\label{asymptotic-MA}
For $k$ sufficiently large, we have $$\sum_i b_i\langle \L+\scO_{\pr^1}(k) \rangle^n\cdot E_i = \langle L\rangle^n$$ where the positive intersection products further satisfy for any $j>0$
$$\langle \L+\scO_{\pr^1}(k+j)\rangle^n\cdot E_i = \langle \L+\scO_{\pr^1}(k)\rangle^n\cdot E_i.$$
\end{lemma}

\begin{proof}
By birational invariance of volume and Equations \eqref{tess-1} and \eqref{tess-2}, to prove the claimed equalities, it suffices to prove that they hold on an associated dominant degeneration, so we may assume $\X$ is dominant. We claim that the result holds provided $\L - L +\scO_{\pr^1}(k)$ is effective and $\L+\scO_{\pr^1}(k)$ is big, which is automatic for $k \gg 0$ by Remark \ref{rmk:big}. 

We begin with the first equality, where we assume $\L-L$ is effective and $\L$ is big. Since $\cX_0=\sum b_i E_i$ and the positive intersection product is numerical, as $\L$ is big it follows that \begin{align*}\sum_i b_i\langle \L\rangle^n\cdot E_i &= \langle \L^n\rangle \cdot \X_0, \\ &= \langle \L^n\rangle \cdot \X_1,\end{align*} which in turn equals $\langle L\rangle^n$ by the proof of Proposition \ref{trusiani}.

To prove the second equality, we again assume $\L-L$ is effective and $\L$ is big. For each $i$ it then follows that for any $k>0$ $$b_i\langle \L+\scO_{\pr^1}(k)\rangle^n\cdot E_i \geq b_i\langle \L\rangle^n\cdot E_i,$$ by monotonicity of the positive intersection product. Since these quantities are nonnegative for all $i$ and sum to $\langle L\rangle^n$, it follows that they must be equal for all $k$.

\end{proof}

This allows us to associate a divisorial measure to a test configuration.

\begin{definition}
    Given a test configuration $(\cX,\cL)$ with $\cX_0=\sum b_i E_i$ and $\cL$ satisfying the conclusion of Lemma \ref{asymptotic-MA}, we define its \textit{Monge--Amp\`ere measure} to be
    $$\MA(\cX,\cL):=\{E_i, \vol(L)\mi b_i\langle \L\rangle^n\cdot E_i\},$$ which is a divisorial measure.
\end{definition}

\begin{remark}\label{rmk:why-na}
Provided one endows $X^\div$ with the structure of a topological space, the Monge--Amp\`ere measure can be viewed as a (discrete) measure on  $X^\div$, simply by defining it to be $\sum_i  \vol(L)\mi (b_i\langle \L\rangle^n\cdot E_i) \delta_{E_i}$. Similarly, a divisorial measure can be viewed as a measure on $X^\div$, and we use this terminology (due to Boucksom--Jonsson), although we emphasise we do not require any topological aspects of  $X^\div$.
\end{remark}

\subsection{The Calabi--Yau theorem}

We next produce a test configuration with given Monge--Amp\`ere measure.

\begin{theorem}\label{thm:main} Let $\mu$ be a divisorial measure, and suppose $\X$ is $\mu$-compatible and dominant. Then there exists $\cL$ on $\cX$ such that   $$\MA(\cX,\cL)=\mu.$$
\end{theorem}

To prove Theorem \ref{thm:main}, we adapt to the big case the ideas of Witt Nystr\"om \cite{wn:icm} (who considered the setting of a K\"ahler class). The idea of the proof is similar to the Brouwer fixed point theorem, whereby we reduce the problem to a problem concerning maps between simplices, and conclude using topological properties of maps between spheres (or rather simplices in our setting). In order to produce a map to which this idea applies, we first combinatorially construct certain vertical divisors $D_i$ in $\cX$.

Fix a dominant, $\mu$-compatible degeneration $\X$, through Lemma \ref{lem:rees}. Define $\Ver(\X)$ to be the vector space of $\R$-Cartier divisors on $\X$ with support in $\X_0$, modulo the addition of $\scO_{\pr^1}(k) = k\X_0$ for all $k\in\R$. Thus for each $D\in\Ver(\X)$, we obtain a test configuration $(\X,L+D)$, and test configurations with underlying degeneration $\X$ precisely correspond to elements of $\Ver(\X)$, up to the addition of $\scO_{\pr^1}(k).$ 

We again denote $\cX=\sum b_i E_i$, for $i \in I$. Lemma \ref{asymptotic-MA} then implies that for any $D \in \Ver(\X)$, the restricted volumes $\langle L+D\rangle^n.E_i$ are well-defined, once one adds a sufficiently large multiple of $\scO_{\pr^1}(k)$. Similarly, we may assume that the representative of the equivalence class that we take is big. 

\begin{proposition}\label{prop:cont}
   For each $i\in I$, the function $\Ver(\X) \to \R$ defined by
    $$D\mapsto \langle L+D\rangle^n.E_i$$
    is continuous.
\end{proposition}
\begin{proof}

When the line bundle $L+D$ is $E_i$-pseudoeffective, since we may assume $L+D$ is big, the claim follows from \cite[Corollary 4.11, Theorem 4.15]{bfj:teissier}. Otherwise, the cone of classes $L+D$ which are not $E_i$-pseudoeffective is closed, and the quantity $\langle L+D\rangle^n.E_i$ vanishes identically in this cone \cite[Theorem 4.15]{bfj:teissier}, proving the desired continuity. \end{proof}

We may now construct the relevant divisors, following Witt Nystr\"om \cite{wn:icm} in the ample case.
\begin{proposition}\label{prop:div}
There exists a collection $D_i\in \Ver(\X)$ such that, for each $k$, any convex combination of the $L+D_i$ with $i\neq k$ is not $E_k$-pseudoeffective.
\end{proposition}

\begin{proof}

The construction is identical to   \cite[Lemma 6.2]{wn:icm}, so we merely recall the definition of the $D_i$, as the proof of the claimed properties is unchanged from the ample case. We consider the intersection graph $\cX_0$, which has vertices for each component $E_i$ and edges when two components intersect. We denote by $d(i,j)$ the natural (discrete) metric on this intersection graph, and define for each $i$ and $t>0$     $$D_{i,t}:=\sum_{j\neq i} \left(1+t^{|I|}-t^{|I|-d(i,j)}\right)b_j E_j.$$ There is then a $t\gg 0$ such that $D_{i,t}$ have the claimed property, essentially as a consequence of the equality $$D_{i,t}|_{E_k} = -t^{|I|+1-d(i,k)}\sum_{j\textrm{ with } d(i,k)-d(i,j)=1}a_j E_j|_{E_k}+O\left(t^{N-d(i,k)}\right),$$ since the leading coefficient of $t$, namely $-a_j E_j|_{E_k}$, is not itself $E_k$-pseudoeffective.

\end{proof}

We may now prove the Calabi--Yau theorem, again following \cite{wn:icm}:
\begin{proof}[Proof of Theorem \ref{thm:main}]
Let $\Delta$ be the unit simplex in $\bbr^N$ with $N:=|I|$, and fix the $D_i$ provided by Proposition \ref{prop:div}. We consider the map $f: \Delta \to \Delta$ defined for $t \in \Delta \subset \R^N$ by $$f(t) = \left\{ \vol(L)\mi b_0\langle L+D_t\rangle^n\cdot E_0), \ldots, \vol(L)\mi b_N\langle L+D_t\rangle^n\cdot E_0)\right\},$$ where $D_t = \sum_i t_i D_i;$ that $f(t)$ lies in $\Delta$ follows from the fact that the Monge--Amp\`ere measure is divisorial.

We claim that $f$ maps  each component of the boundary of $\Delta$ to itself; that is, if $t\in \Delta$ satisfies $t_i=0$ for some $i\in I$, then $\langle L+D_t\rangle^n\cdot E_i=0$.  But since $D_i=0$
$$
       (L+D_t)|_{E_i}=\sum_{j\neq i}t_j (\pi^*L+D_j)|_{E_i},
    $$
    since $\sum_{j\in I} t_j=\sum_{j\neq i} t_j=1$. By definition of the $D_i$ in Proposition \ref{prop:div}, this convex sum is not $E_j$-pseudoeffective on $E_j$, so that by  \eqref{eq:restrineq}
    $$0\leq\langle L+D_t\rangle^n\cdot E_i\leq \langle (L+D_t)|_{E_j}^n\rangle=0,$$
which is what we were required to prove.

From here, surjectivity of $f$ follows for topological reasons. Firstly, it follows from Proposition \ref{prop:cont} that $f$ is continuous.  Further, $f$ is homotopic to the identity through the explicit homotopy $f_{\lambda}(t) = (1-\lambda)f(t) +t\lambda$. Since $f$ in addition maps the boundary $\partial\Delta$ to itself, it follows that $f|_{ \partial \Delta }: \partial \Delta \to  \partial \Delta $ is homotopic to the identity, hence has degree one as a map between simplices. But every continuous map $f: \Delta \to \Delta$ with $f|_{ \partial \Delta }: \partial \Delta \to  \partial \Delta $ of nonzero degree is surjective (supposing there is a $z\in \Delta \setminus f(\Delta)$, the map $F: \Delta \to \partial \Delta$ defined by $F(v) = (f(v) - z)/|f(v)-z|$ would  induce a continuous map $F|_{\partial \Delta}:  \partial \Delta \to  \partial \Delta$ which extends over $\Delta$, implying $F|_{\partial \Delta}$ must have degree zero, contradicting that $F|_{\partial \Delta}$ is homotopic to $f|_{ \partial \Delta }$). \end{proof}

\subsection{A variational characterisation.} 
We next turn to variational aspects, by characterising solutions of the Calabi--Yau theorem realised on a given test configuration as maximisers of a functional involving the Monge--Ampère energy. We begin by proving a concavity property, for which we assume $\X$ is dominant.

\begin{proposition}\label{prop:concave}

The functional $E^\na: \Ver(\cX) \to \R$ defined by $$E\NA(D) = E\NA(\X,L+D)$$ is concave.

\end{proposition}

We prove this result by reducing to the case where $L$ is ample, by a Fujita approximation argument. We thus begin by establishing the following.

\begin{lemma}\label{prop:fujitatc} 
    Let $\mu:\cX\to X$ be a test configuration for $X$, and let $D, D'\in \Ver(\cX,L)$. For all $\varepsilon>0$ there exists an $m>0$ and
    \begin{enumerate}[(i)]
    		\item $\pi_m: X_m\to X$ a birational model for $X$;
    		\item $L_m$ a semiample line bundle on $X_m$ with $\pi_m^*L-L_m$ effective;
        \item $\mu_m:\cX_m\to\cX$ a dominant  degeneration of $X_m$ compatible with $\pi_m$;
        \item $D_m, D_m' \in \Ver(\cX_m)$  are such that $\pi_m^*L+D_m$, $\pi_m^*L+D_m'$ are relatively semiample on $\X$ and $\mu_m^*D- D_m$ and $\mu_m^*D- D_m$ are effective,
   	\end{enumerate}
   	with
	 $$\left|E^\na(\cX_m,\pi_m^*L+D_m) - E^\na(\cX,\pi^*L+D)\right|<\varepsilon,
	$$ and similarly for $D'$.\end{lemma}

\begin{proof}
The existence of $(X_m,L_m)$ themselves satisfying the desired conditions is a consequence of the Fujita approximation theorem, as in Equation \eqref{eq:fujita}, where explicitly we define $\pi_m:X_m=\Bl_{\Bs(mL)}X\to X$ and $L_m:=\pi_m^*L-m\mi E_m$, where $E_m$ is the exceptional divisor and $m \gg 0$.

The result is unchanged by adding $\scO_{\pr^1}(k)$ to $D$ or $D'$, so we may assume $D$ and $D'$ are effective and $L+D$ and $L+D'$ are big.  We then similarly define
	$$\mu_m:\cX_m=\Bl_{\Bs(m(L+D))}\cX\to \cX$$
	with exceptional divisor $\cE_m$, and similarly define $\X_m'$. It follows from  Proposition \ref{trusiani}, that $(\X_m, \mu_m^*(\pi_m^*L + D) - m^{-1}\E_m)$ and  $(\X_m', \mu_m^*(\pi_m^*L + D') - m^{-1}\E_m' )$ are test configurations for $(X_m,L_m)$. As the result is unchanged by passing to a third degeneration of $X_m$ which dominates $\X_m$ and $\X_m'$, we may assume that  $\X_m=\X_m'$. The resulting line bundles on $\X_m$ are therefore semiample.
	
	As we have produced Fujita approximations for $(\X,L+D)$ and $(\X,L+D')$ compatible with Fujita approximations of $(X,L)$ itself, the conclusion concerning the Monge--Amp\`ere energies is a simple consequence, as 
 $\L_m^{n+1} \to \langle L+D\rangle^{n+1}$ and $L_m^n \to \langle L\rangle^n$. To conclude, we note that $\mu_m^*D- D_m$ is effective as it simply equals $m^{-1} \cE_m$, and as the same is true of $\mu_m^*D'- D_m'$, the result follows.  \end{proof}
 
We may now prove concavity.

\begin{proof}[Proof of Proposition \ref{prop:concave}]

For general ample line bundles, concavity of the Monge--Amp\`ere energy is a simple consequence of the Hodge index theorem \cite[Proposition 7.7]{bj:trivval}; by taking second derivatives, the claim reduces to the fact that when $D$ is vertical and $\L$ is a general relatively semiample line bundle on degeneration $\X$ of a projective variety $X$, we have $\L^{n-1}.D^2 \leq 0$, which is proven for example in Li--Xu \cite[Lemma 1]{lixu}. The conclusion thus holds when the line bundle on $X$ is merely semiample, either by continuity or by following the same line of proof.

We argue by reducing to the semiample case. For  $D,D'\in \Ver(\cX)$, we must prove $$E^\na(\cX,L+(1-t)D+tD') \geq  (1-t)E^\na(\cX,L+D)+tE^\na(\cX,L+D').$$ Fix $\epsilon>0$, and let $\X_m,D_m,D_m'$ be the output of Lemma \ref{prop:fujitatc}. Effectivity of $\mu_m^*D- D_m$ and  $\mu_m^*D- D_m$ implies $$ E^\na(\cX,L+(1-t)D+tD')\geq E^\na(\cX_m,L+(1-t)D_m+tD_m'),$$ by definition of the Monge--Amp\`ere operator as a volume, along with birational invariance of volume. Using this and Lemma \ref{prop:fujitatc} we see \begin{align*} 
       E^\na(\cX,L+(1-t)D+tD')&\geq E^\na(\cX_m,L+(1-t)D_m+tD_m'),\\
        &\geq (1-t)E^\na(\cX_m,L+D_m)+tE^\na(\cX_m,L+D'_m),\\
        &\geq (1-t)E^\na(\cX,L+D)+tE^\na(\cX,L+D')-\varepsilon,
\end{align*} where we also used the concavity of the Monge--Amp\`ere functional in the semiample case noted above. As this holds for all $\epsilon>0$, the result follows.\end{proof}

We may now characterise solutions of the Monge--Amp\`ere equation. We consider a divisorial measure $\mu = \{(F_j,\xi_j)\}_{j\in J}$, and fix a dominant, $\mu$-compatible degeneration $\X$ of $X$ with $\X_0 = \sum_{i\in I}b_iE_i$, so that $J\subseteq I$. For each $t\in \Ver(\X) \cong \R^{|I|}$ we thus obtain a divisor $D_t = \sum_i b_i t_i E_i \in \Ver(\X)$ (we note that we incorporate the $b_i$ into $D_t$, in contrast with our prior convention).

\begin{theorem}\label{thm:variationalE}
	The test configuration $(\cX,L+D_t)$ solves the Monge--Amp\`ere equation $\MA(\cX,L+D_t)=\mu$ if and only if
		\begin{equation}
			E^\na(\cX, L+D_t)-\langle \xi,t\rangle=\sup_{s\in  \Ver(\X)} \left\{E^\na(\cX,L+D_s)-\langle \xi,s\rangle\right\}.
		\end{equation}
Here, $\langle \xi,t\rangle$ is understood as $\langle \xi,\pi_{I,J}(t)\rangle=\sum_{j\in J} \xi_j t_j$.
\end{theorem}

Note $E^\na(\cX,L+D_t+\scO_{\pr^1}(k))=E^\na(\cX,L+D_t)+k$, and $\langle \xi,t+k\rangle = \langle \xi,t\rangle + \sum_j \xi_j k = \langle \xi,t\rangle+k$, since $\mu$ is a divisorial measure. The functional $t\to E^\na(\cX, L+D_t)-\langle \xi,t\rangle$  can thus be viewed as a functional on $\Ver(\X)$, which is implicit in the statement. For the same reason, in proving the result, we may assume the $L+D_t$ involved are big. 

\begin{proof}

The functional $s \mapsto E^\na(\cX,L+D_s)-\langle \xi,s\rangle$ is concave, by concavity of $s \mapsto E^\na(\cX,L+D_s)$ and linearity of $s \mapsto \langle \xi,s\rangle.$ It is further differentiable on $\Ver(\X)$, by differentiability of volume on the big cone of a projective variety,  the fact that we may assume $L+D_s$ is actually big, and the definition of the Monge--Amp\`ere energy. It follows that its supremum is achieved at critical points, and that this supremum is independent of the critical point considered. 

It therefore suffices to prove that critical points are precisely solutions of the  Monge--Amp\`ere equation. We calculate for $\delta D = \sum_i (\delta t_i)E_i$ that $$\frac{d}{dt}\bigg|_{t=0}E^\na(\cX,L+D+t\delta D) = \vol(L)^{-1}\sum_i (\delta t_i)b_i\langle L+D\rangle^n\cdot E_i-\langle \xi, \delta t\rangle;$$ if this vanishes for all $\delta t_i$, it follows that $\vol(L)^{-1}b_i\langle L+D\rangle^n\cdot E_k = \xi_j$ for each $j$, and $\vol(L)^{-1}b_i\langle L+D\rangle^n\cdot E_k=0$ for $i\in I-J$, which implies the result. \end{proof}

The important aspect of this result is the following: while we do not expect (and do not prove) uniqueness of solutions of the Monge--Amp\`ere equation, each solution nevertheless achieves the same value of the functional $t\mapsto E^\na(\cX,L+D_t)-\langle \xi,t\rangle$ on $\R^{|I|}$.

\section{Divisorial stability and K-stability}\label{sect:3}

We now aim to show that uniform K-stability and divisorial stability of $(X,L)$, with $X$ smooth and $L$ big, are equivalent, through the Calabi--Yau theorem. This involves developing results around filtrations, most notably \textit{divisorial} filtrations. We will use these ideas to prove a result relating  the Monge--Ampère energy of a test configuration to the volume of the associated divisorial filtration. This, along with left differentiability of the norm of a measure, are the key steps towards proving the equivalence of  uniform K-stability and divisorial stability. We end the section with an application, by relating K-stability in the case $L=-K_X$ to prior work \cite{derreb:1, dar:zhangbigke}.

\subsection{Filtrations on the algebra of sections of a big line bundle}

Let us denote by $R(X,L)=\bigoplus_k R_k$ the section ring of $L$, with $R_k:=H^0(X,kL)$. A \textit{filtration} on $R(X,L)$ is the data, for each $\lambda\in\bbr$ and $k\in\mathbb{N}$, of a vector space $F^\lambda R_k\subseteq R_k$ such that:
\begin{enumerate}[(i)]
    \item $\cF^\lambda R_k\subseteq \cF^\mu R_k$ for $\lambda\geq \mu$;
    \item $\cF^\lambda R_k=\bigcap_{\mu<\lambda} \cF^\mu R_k$;
    \item there exists $e_-\in \bbr$ such that $\cF^\lambda R_k=R_k$ for all $\lambda<ke_-$;
    \item there exists $e_+\in\bbr$, $e_-\leq e_+$, such that $\cF^\lambda R_k=\{0\}$ for all $\lambda>ke_+$;
    \item for all $k,\ell$ and $\lambda,\mu$, $\cF^\lambda R_k\cdot \cF^\mu R_\ell\subseteq \cF^{\lambda+\mu}R_{k+\ell}$.
\end{enumerate}
The definitions imply that, for $k$ fixed, the dimension function
$$\lambda\mapsto \dim \cF^\lambda R_k$$
is piecewise constant, non-increasing, left-continuous, equal to $\dim R_k$ for $\lambda$ sufficiently small (uniformly in $k$) and to zero for $\lambda$ sufficiently large (likewise). It therefore has exactly $\dim R_k$ discontinuity points when counted with multiplicity, where the multiplicity at a discontinuity point $\lambda$ equals the positive integer $\dim \cF^\lambda R_k - \dim\cF^{\lambda+\varepsilon}R_k$ for any sufficiently small $\varepsilon>0$. This yields the naturally ordered set of \textit{jumping values} $\lambda_{k,i}(\cF_k)$ of $\cF R_k$. The multiplicativity condition $(v)$ provides good equidistribution properties of the jumping numbers as $k\to\infty$, while the boundedness conditions $(iii)$ and $(iv)$ ensure that the jumping numbers equidistribute to a compactly supported measure on $\bbr$. As a special case of those equidistribution properties, one can define the \textit{volume} $\vol(\cF)$ of the filtration $\cF$ to be the normalised limit
\begin{equation}\label{eq:vol1}
\vol(\cF):=\lim_{k\to\infty}k\mi \vol(\cF_k),
\end{equation}
where
\begin{equation}\label{eq:vol2}
\vol(\cF_k):=(\dim R_k)\mi \sum_i \lambda_{k,i}(\cF_k),
\end{equation}
which was shown to converge when $L$ is a big line bundle by \cite[Theorem A]{bouchen}.

We will be interested in explicit ways to compute this volume, which we do using volumes of subalgebras of $R(X,L)$. To that end, we first define the graded algebra
$$R(X,L)_{\geq \lambda,\cF}:=\bigoplus_k \cF^{k\lambda}R_k.$$
By \cite[Lemma 3.30]{book:xukstab} and \cite{bouchen}, we have
\begin{equation}\label{eq:bchen}
\vol(\cF):=e_-+\int_{e_-}^\infty \vol(R(X,L)_{\geq \lambda,\cF})\,d\lambda
\end{equation}
where
$$\vol(R(X,L)_{\geq\lambda,\cF}):=\lim_{k\to\infty}k^{-n}\dim \cF^{k\lambda}R_k$$
is a well-defined limit \cite[Corollary 1.11]{bouchen}, and $e_-$ is any real number satisfying condition $(iii)$ above. Note that this volume is zero for $\lambda>e_+$, with $e_+$ given by condition $(iv)$, so that this expression often appears integrated up to $e_+$ in the literature. Equivalently, in the sense of distributions, setting
$$\nu(\cF):=-\vol(L)\mi \frac{d}{d\lambda}\vol(R(X,L)_{\geq \lambda,\cF}),$$
we have
\begin{equation}\label{eq:bchen2}
	\vol(\cF)=\int_\bbr \lambda d\nu(\cF).
\end{equation}

We say that a basis $\{s_i\}$ of $R_k$ is \textit{adapted} to the filtration $\cF$ if, for all $\lambda\in\bbr$, $\mathrm{Span}\left\{s_i\in \cF^\lambda R_k\right\}=\cF^\lambda R_k$. Given such a basis, then one may describe the jumping values as
\begin{equation}\label{eq:succmin}
	\lambda_{k,i}(\cF_k)=\sup\{\lambda\in\bbr,\,s_i\in \cF^\lambda R_k\},
\end{equation}
which we order so that they form a nonincreasing sequence in $i$. 

Given now two filtrations $\cF$, one can find for each $k$ a basis $\{s_i\}_i$ of $R_k$ that is simultaneously adapted to $\cF_k$ and $\cF'_k$, and define their \textit{relative jumping numbers} as the (reordered to be nonincreasing) sequence
\begin{equation}\label{eq:rel-jump}\lambda_i(\cF_k,\cF'_k):=\sup\left\{\lambda\in\bbr,\,s_i\in \cF'{}^\lambda R_k\right\}-\sup\left\{\lambda\in\bbr,\,s_i\in \cF^\lambda R_k\right\}.\end{equation}
One can then define a genuine \textit{distance} between the two filtrations \cite{goldmaniwahori} as
$$d_\infty(\cF_k,\cF'_k):=\max_i |\lambda_{k,i}(\cF_k,\cF'_k)|.$$
Importantly, this distance controls the volume: by \cite[Proposition 2.14 $(iv)$]{boueri}, we have
$$
	|\vol(\cF_k)-\vol(\cF'_k)|\leq \dim R_k \cdot d_\infty(\cF_k,\cF'_k).
$$
In particular, from the definition of the volume of $\cF$ we obtain
\begin{equation}\label{eq:d1dinfty}
	|\vol(\cF)-\vol(\cF')|\leq \limsup_{k\to\infty}k\mi d_\infty(\cF_k,\cF'_k).
\end{equation}

\subsection{The volume of divisorial filtrations}

Given a collection $\{E_i,t_i\}$, we may define a filtration of $R(X,L)$ by
\begin{equation}\label{eq:divfiltr}
    \cF^\lambda R_k=\{s\in R_k,\,\min_i\{\ord_{E_i}(s)+kt_i\}\geq \lambda\}.
\end{equation}
We call such a filtration a \textit{divisorial} filtration. A special case is that of a single valuation with $t=0$, in which case one recovers the usual valuative filtrations considered in e.g.\ \cite{blumjonsson}. In the case where $L$ is ample, some of the material in this section appears in \cite{bj:kstab1}.

We next prove a selection of results concerning the volume of divisorial filtrations. We first provide an  explicit expression of a divisorial filtration. In what follows, we fix a collection $\{E_i,t_i\}_{i\in I}$ where either $E_i$ is a rational positive multiple of a prime divisor over $X$, or is the trivial valuation on $X$ (which we denote $E_{\triv}$). By convention, we ask that $0\in I$ if and only if $E_0 = E_{\triv}$ is the trivial valuation. This input produces a divisorial filtration  of $R(X,L)$, which we denote $\cF:=\cF_{S,t}$. Let $\gamma_i=\gamma(L,E_i)$ be the pseudoeffective threshold of $E_i$ with respect to $L$, where we set $\gamma(E_{\triv}) = 0$.

\begin{theorem}\label{thm:descrfiltr} The divisorial filtration $\cF$ satisfies
\begin{enumerate}[(i)]
	\item $R(X,L)_{\geq \lambda,\cF}=\{0\}$ for $\lambda>\lambda_{\max}:=\min_{i\in I}(\gamma_i + t_i)$;
	\item $R(X,L)_{\geq \lambda,\cF}=R(X,L)$ for $\lambda<\lambda_{\min}:=\min_{i\in I} t_i$;
	\item $R(X,L)_{\geq \lambda,\cF}=R\left(X,L-\sum_{i\in I-\{0\}}\max\{\lambda-t_i,0\}E_i\right)$ for $\lambda_{\min}\leq \lambda\leq\lambda_{\max}$.
\end{enumerate} In particular,
$$\vol(\cF)=e_-+\vol(L)\mi\int_{e_-}^\infty\vol\left(L-\sum_{i\in I-\{0\}}\max\{\lambda-t_i,0\}E_i\right)\,d\lambda$$
for any $e_-\leq \lambda_{\min}$.
\end{theorem}
\begin{proof}
    The divisorial filtration is defined by 
    \begin{align*}
    		\cF^{k\lambda} R_k=\{s\in R_k,\min_{i\in I}\{\ord_{E_i}(s)+kt_i\}\geq \lambda\}.
    	\end{align*} 	If $0\in I$, so that $E_0$ is the trivial valuation, the condition $\ord_{E_{\triv}}(s)+kt_i\geq \lambda$ holds for all $s\in R_k$ when $\lambda\leq kt_0$, and only for $0\in R_k$ when $\lambda>kt_0$.

	For $i\neq 0\in I$, by definition of the pseudoeffective threshold, if $\gamma>\gamma_i$ then $R(X,L-\gamma E_i)=\{0\}$. The condition $\ord_{E_i}(s)+kt_i\geq k\lambda$ may be rewritten as $s\in H^0(kL-k(\lambda-t_i)E_i)$, hence:
    \begin{enumerate}[(i)]
        \item if $\lambda-t_i\leq 0$, then the condition $\ord_{E_i}(s)+kt_i\geq \lambda$ is automatic; 
        \item if $0\leq \lambda-t_i\leq \gamma_i$, then $\lambda-t_i$ is a candidate for the supremum defining $\gamma_i$, and in particular the condition $\ord_{E_i}(s) + kt_i\geq \lambda$ is nontrivial;
        \item if $\lambda-t_i>\gamma_i$ then the condition $\ord_{E_i}(s) + kt_i$ is satisfied only for $0\in R_k$.
    \end{enumerate}

    In summary, it follows that $\cF^{k\lambda} R_k=\{0\}$ for $\lambda>\min_{i\in I}\{\gamma_i + t_i\}$, while $\cF^{k\lambda} R_k=R_k$ for $\lambda<\min_{i\in I} t_i$, and for $\lambda$ in-between those bounds
    \begin{equation}\label{eq:sectionsfiltration}
    		\cF^{k\lambda} R_k=H^0\left(X,kL-k\sum_{i\in I-\{0\}}\max\{\lambda-t_i,0\}E_i\right).
    	\end{equation}
  It follows from Equation \eqref{eq:bchen} that for any $e_-< \lambda_{\min}$
    \begin{align*}
        \vol(\cF)&=e_-+\vol(L)\mi\int_{e_-}^\infty \vol(R(X,L)_{\geq \lambda,\cF})\,d\lambda,\\
        &=e_-+\vol(L)\mi\int_{e_-}^\infty\vol\left(L-\sum_{i\in I-\{0\}}\max\{\lambda-t_i,0\}E_i\right)\,d\lambda,\\
        &=e_-+\vol(L)\mi\int_{e_-}^\infty\vol\left(L-\sum_{i\in I-\{0\}}\max\{\lambda-t_i,0\}E_i\right)\,d\lambda.
\end{align*} Continuity of volume to the boundary of the big cone further implies we may take $e_- = \lambda_{\min}$.\end{proof}

\begin{remark}\label{rmk:filtr}
	One sees that only the index set $\{i\in I,t_i\leq \lambda_{\max}\} \subset I$ contributes to constraints defining the filtration in the statement of Theorem \ref{thm:descrfiltr}, so that
$$\vol(\cF)=e_-+\vol(L)\mi\int_{e_-}^\infty\vol\left(L-\sum_{i,\,t_i\leq\lambda_{\max}}\max\{\lambda-t_i,0\}E_i\right)\,d\lambda.$$
In particular, if the trivial valuation belongs to the collection  $\{E_i,t_i\}$, i.e.\ $0\in I$, then the indices with $t_i>t_0$ do not contribute.
\end{remark}

\begin{example}\label{ex:tc}
A test configuration $(\cX,\cL)$ canonically induces a divisorial filtration, a construction going back to Witt Nystr\"om \cite{wn:tcoko}: the associated filtration $\cF^\lambda R_k$ is the set of sections $s\in R_k$ such that $t^{-\lambda}\tilde s$ extends to a regular section across $\cX_0$, where $\tilde s$ is the rational section of $k\cL$ induced by the $\bbc^*$-action. This filtration can be described concretely as follows: if $\cX_0=\sum b_i E_i$, and $\cL=\pi^*L+D$, then
$$\cF^\lambda R_k=\{s\in R_k,\,\min_i\{b_i\mi\ord_{E_i}(s)+kb_E\mi \ord_{E_i}(D)\}\geq \lambda\},$$
see \cite[Lemma 5.17]{bhj:duistermaat} and \cite[Proposition A.3]{bj:kstab1}. The proper transform of $X\times \{0\}$ under the birational map $\X \dashrightarrow X\times\pr^1$ produces a canonical component of $\X_0$, which induces the trivial valuation $E_{\triv}$. 

If $\ord_{E_i}(D)\geq 0$ for all $i$, then $\cF$ is a shift of the trivial filtration by Remark \ref{rmk:filtr}. Thus for example, if $t_0=0$ and $t_i>t_0=0$ for all $i\neq 0$, the associated filtration is the trivial filtration. 	Conversely, if $L$ is ample and $\cL$ is relatively semiample, it is automatic from \cite[Theorem 5.16]{bhj:duistermaat} that $t_i\leq t_0$ for all $i$.

\end{example}

We define $\cF^\div$ to be the set of divisorial filtrations on $X$. If $S\subset X^\div$ is a finite subset, we further write $\cF^\div_S$ for the set of divisorial filtrations involving only the divisorial valuations in $S$. Given $t\in \bbr^{|S|}$, we denote by $\cF_{S,t}$ the associated filtration
$$\cF_{S,t}^\lambda R_k:=\{s\in R_k,\,\min_{\nu_i\in S}\{\nu_i(s)+kt_i\}\geq \lambda\}.$$ 

\begin{remark}We emphasise  that there is  no canonical way to describe the coefficients of a given $\cF=\cF_{S,t}\in \cF_S$: a filtration does not canonically induce a collection $\{E_i,t_i\}$. For example, if  $E_i\in S$ does not contribute to the  constraints defining the filtration $\cF_{S,t}$, then for all $c>0$, $t'=(t_1,\dots,t_i+c,\dots,t_{|S|})$ satisfies $\cF_{S,t}=\cF_{S,t'}$.\end{remark}

As a consequence, we obtain the following, which will be essential to a compactness argument  in the proof of Proposition \ref{prop:d1dinfty}. A crucial aspect is that the bound we obtain depends purely only on $S$, and not on the given filtration.
\begin{corollary}\label{coro:compactness}
    Suppose $\cF$ is a divisorial filtration with $\lambda_{\min}=0$. If $\cF$ is nontrivial, then for any divisorial filtration  $\cF_{S,t} = \cF$ with $S\subset X^\div$  and $t\in \R^{|S|}$, we may find a $t'\in \bbr^{|S|}$ satisfying $$t'_i\leq \max_{i} \gamma(L,E_i)+1$$  for all $i$, such that   $\cF=\cF_{S,t'}$.\end{corollary}
\begin{proof}
	
	By Theorem \ref{thm:descrfiltr}, $\lambda_{\min}=\min_i t_i$. The hypothesis that $\lambda_{\min}=0$ thus implies that, for any $j$ with $t_j = 0$,   the quantity $\lambda_{\max}$ satisfies $$\lambda_{\max}=\min_{i}\{\gamma_i+t_i\}\leq \gamma_j.$$ In particular, $\lambda_{\max}\leq \max_{i}\gamma(L,E_i)$.
	
	On the other hand, for an arbitrary divisorial filtration (not necessarily satisfying $\lambda_{\min}=0$), Theorem \ref{thm:descrfiltr} implies  that the indices $j$ with $t_j> \lambda_{\max}$ do not contribute to the constraints defining the filtration, by Remark \ref{rmk:filtr}. We therefore may replace $t_j$ with any  $t'_j$ satisfying $\lambda_{\max}<t'_j<\lambda_{\max}+1$, without changing the resulting filtration. Combining this with the previous paragraph proves the result.
\end{proof}

A similar idea produces the following.

\begin{lemma}\label{lem:incl} For any  $S\subseteq T \subset X^\div$, any divisorial filtration of the form $\cF_{S,s}$ can be extended to a divisorial filtration of the form $\cF_{T,t}$, where the $S$-components of $t$ equal  $s$, without changing the resulting filtration.
\end{lemma}
\begin{proof}
	For any $s\in \bbr^{|S|}$ with $\cF=\cF_{S,s}$, we may  extend $s$ to some $t\in \bbr^{|T|}$ such that its coefficients corresponding to $T\setminus S$ are larger than $\max_{i} \{t_i+ \gamma(L,E_i)\}$. With this choice, it follows from Theorem \ref{thm:descrfiltr} that  $\cF_{S,s}=\cF_{T,t}$.\end{proof}

 Conversely, for $S\subseteq T \subset X^\div$ finite, we may produce a restriction map $\cF^\div_T\to \cF^\div_S$, simply by considering the projection  $\pi: \R^{|T|}\to \R^{|S|}$ onto the subspace $\R^{|T|}\to \R^{|S|}$. While this is not well-defined at the level of filtrations (but rather well-defined at the level of representations of divisorial filtrations in the form $\cF_{T,t}$ for some $t$, such representations being nonunique for a given divisorial filtration), we record the following volume inequality, which uses this notation.

\begin{lemma}\label{lem:res}

For $t\in \bbr^{|T|}$, the volume of the resulting divisorial filtration $\cF_{T,t}$ satisfies $$ \vol(\cF_{T,t}) \leq\vol(\cF_{S,\pi(t)}).$$
\end{lemma}
\begin{proof}
	Since $S\subseteq T$ we have
    \begin{align*}
    		\cF_{T,t}^{\lambda} R_k&=\{s\in R_k,\,\min_{i \in T}\{\ord_{E_i}(s)-kt_i\}\geq \lambda\}\\
    		&\subseteq \{s\in R_k,\,\min_{j \in S}\{\ord_{E_i}(s)-kt_i\}\geq \lambda\}\\
    		&=\cF^{\lambda}_{S,\pi(t)} R_k.
    	\end{align*}
    for all $\lambda$, and hence $\vol(R(X,L)_{\geq \lambda,\cF_{T,t}})\leq\vol(R(X,L)_{\geq \lambda,\cF_{S,\pi(t)}})$. The result thus follows from Equation \eqref{eq:bchen}.
\end{proof}

Finally, we will use the following result on continuity of the volume upon changing the coefficients of a divisorial filtration.

\begin{proposition}\label{prop:d1dinfty}
	For $S\subset X^\div$  and any  $t,t'\in \bbr^{|S|}$, we have
	$$|\vol(\cF_{S,t})-\vol(\cF_{S,t'})|\leq \norm[t-t']_\infty.$$
\end{proposition}
\begin{proof}
	We claim that for each $k$, and any $t,t'\in \bbr^{|S|}$,
	$$\max_i |\lambda_{k,i}(\cF_{S,t},\cF_{S,t'})|\leq k\norm[t-t']_\infty,$$ where the $\lambda_{k,i}(\cF_{S,t},\cF_{S,t'})$ are the relative jumping numbers defined in Equation \eqref{eq:rel-jump}. Indeed, choosing a basis $\{s_i\}$ of $R_k$ compatible with both filtrations, 
	\begin{align*}
		\lambda_{k,i}(\cF_{S,t},\cF_{S,t'})&=\sup\left\{\lambda,\,s_i\in \cF_{S,t'}^\lambda R_k\right\}-\sup\left\{\lambda,\,s_i\in \cF_{S,t}^\lambda R_k\right\},\\
		&=\sup\{\lambda,\min_{i} \{\ord_{E_i}(s)-kt'_i\}\geq \lambda\}-\sup\{\lambda,\min_{i} \{\ord_{E_i}(s)-kt_i\}\geq \lambda\},
	\end{align*}
	which is clearly absolutely bounded by $k\max_i |t_i-t'_i|$. This having proven the claim, the result then follows from Equation \eqref{eq:d1dinfty}.
\end{proof}

We are now able to give equivalent characterisations of the Legendre transform defining the norm of a divisorial measure (see also \cite[Theorem 7.13]{bj:kstab1} in the ample case).
\begin{proposition}\label{prop:sup}
    Given $\mu=\{E_i,\xi_i\}$ a divisorial measure, denoting $S$ the set $(E_i)_i$, the following three expressions are equal:
    \begin{enumerate}[(i)]
    		\item the norm  $\norm[\mu]_L$ of $\mu$;
        \item the Legendre transform $\sup_{s\in\bbr^{|S|}}\{\vol(\cF_{S,s})-\langle \xi,s\rangle\}$;
        \item for all finite sets $S\subset T\subset X^\div$, the Legendre transform
        $$\sup_{t\in \bbr^{|T|}}\{\vol(\cF_{T,t})-\langle \xi,t\rangle\},$$
        
    \end{enumerate}
	where similarly to Theorem \ref{thm:variationalE} we write $\langle \xi,t\rangle:=\langle \xi,\pi(t)\rangle$ with $\pi:\bbr^{|T|}\to \bbr^{|S|}$ the projection associated with the inclusion $S\subset T$.    
\end{proposition}
\begin{proof}

 By Definition \ref{def:norm-def}, the norm of $\mu$ is given by $$\|\mu\|_L = \sup_{s\in \R^{|S|}}\{ S_{L}(s)-\langle \xi, s\rangle\},$$and so the equality of the first two expressions follows from Theorem \ref{thm:descrfiltr}.
    
Fix $T$ with $S\subset T$. For any $s\in \R^{|S|}$, Lemma \ref{lem:incl} produces a $t\in \R^{|T|}$ with $\pi(t)=s$ and with the filtrations  $\cF_{S,t}$ and $\cF_{T,t'}$ of $R(X,L)$ being equal. The supremum defining $(ii)$ is therefore smaller than the expression in $(iii)$, and to conclude we must prove the reverse inequality. 

Choose a maximising sequence $t_i$ for the expression $(iii)$. Both expressions are unchanged upon translation by a real number, hence we may assume  $\min_i t_i=0$. Furthermore, by Corollary \ref{coro:compactness}, we may also assume that the components of all $t_i$ are bounded above by a constant depending only on $T$. In particular, the $t_i$ converge subsequentially as $t\to\infty$ to some $t_{\infty} \in \R^{|T|}$ with $\pi(t_i)$ converging also in $\bbr^{|S|}$ to $\pi(t_{\infty}) =: s_{\infty}$. We further have:
    \begin{enumerate}[(i)]
        \item $\lim_{i\to\infty}\vol(\cF_{S,\pi(t_i)})=\vol(\cF_{S,s_{\infty}}),$ by Proposition \ref{prop:d1dinfty};
        \item for all $i$, $\vol(\cF_{S,\pi(t_i)})\geq \vol(\cF_{T,t_i})  $, by Lemma \ref{lem:res}.
    \end{enumerate} It follows that
    \begin{align*}
        \vol(\cF_{S,s_\infty})-\langle \xi,s_\infty\rangle&=\lim_{i\to\infty}(\vol(\cF_{S,\pi(t_i)})-\langle \xi,\pi(t_i)\rangle),\\
        &\geq \lim_{i\to\infty}(\vol(\cF_{T,t_i})-\langle \xi,\pi(t_i)\rangle),\\
        &=\sup_{t\in \bbr^{|T|}}\{\vol(\cF_{T,t})-\langle \xi,\pi(t)\rangle\},
    \end{align*}
   	concluding the proof, since $\langle \xi,\pi(t)\rangle = \langle \xi,t\rangle$ by definition (as the $E_i$ correspond to basis vectors in $\R^{|S|}\subset\R^{|T|}$).
\end{proof}

\subsection{Volume and energy.}

We require the following expression of the volume of the filtration associated with a test configuration in terms of the non-Archimedean Monge--Ampère energy. In the case where $L$ is ample, this may be found in \cite{bhj:duistermaat}.
\begin{theorem}\label{thm:bchen}
    Let $(\cX,\cL)$ be a test configuration, and let $\cF$ be its associated filtration of $R(X,L)$. Then
    $$\vol(\cF)=E^\na(\cX,\cL).$$
\end{theorem}
\begin{proof}
	We replace $\cL$ with $\cL+\pi^*\cO_{\bbp^1}(k)$ with $k$ sufficient large so that $\cL$ is big without loss of generality, noting that the desired equality is preserved under this modification.
	
	For $m>0$, let $\tilde \mu_m:\cX_m=\Bl_{\Bs(m\cL)}\cX\to \cX$, and $\mu_m:X_m=\Bl_{\Bs(mL)}X\to X$ be the blowups involved in Fujita approximation. Let $E_m$ be the exceptional divisor of $\mu_m$, and denote $\cE_m$ similarly. We write $L_m:=\mu_m^*(mL)-E_m$, $\cL_m:=\tilde\mu_m^*(m\cL)-\cE_m$, so that $(X_m,m\mi L_m)$ and $(\cX_m,m\mi\cL_m)$ are Fujita approximations of $(X,L)$ and $(\cX,\cL)$ respectively, and $(\X_m,\cL_m)$ is a test configuration for $(X_m,L_m)$ with $\L_m$ big and semiample, by the results of Lemma \ref{prop:fujitatc}.
	
	We now claim that,   denoting by $\cF_m$the filtration induced on $R(X_m,L_m)$ by the test configuration $\cL_m$   for each $m$, we have an identification
	\begin{equation}\label{eq:1}
		\cF^{\lambda}_m H^0(X_m,L_m)\simeq \cF^{\lambda} H^0(X,mL) 
	\end{equation}
	and for all $k>0$, an inclusion 
	\begin{equation}\label{eq:2}
		\cF^\lambda_m H^0(X_m,kL_m)\subseteq \cF^{\lambda} H^0(X,kmL)
	\end{equation}
	via a natural inclusion  $H^0(X_m,kL_m) \hookrightarrow H^0(X,kmL)$. For  Equation \eqref{eq:2}, we choose $k>0$. Let $\tilde \mfb_m$ be the ideal defining $\Bs(m\cL)$, and $\mfb_m$ the ideal defining $\Bs(mL)$. We have
	\begin{align*}H^0(X_m,kL_m)&\simeq^{\tau_k} H^0(X,(kmL)\otimes \mfb_m^k), \\ &\hookrightarrow^{\iota_k} H^0(X,(kmL)\otimes \mfb_{km}), \\ &= H^0(X,kmL),\end{align*}
	and likewise 	
	\begin{align*}H^0(\cX_m,k\cL_m)&\simeq^{\tilde \tau_k} H^0(\cX,(km\cL)\otimes \tilde\mfb_m^k), \\ &\hookrightarrow^{\tilde\iota_k} H^0(\cX,(km\cL)\otimes \tilde\mfb_{km}), \\ &= H^0(\cX,km\cL).\end{align*}
	Under these maps, we have $\tilde \tau_k(t^{-\lambda}\tilde s)=t^{-\lambda}\widetilde{\tau_k(s)}$ while the maps $\tilde \iota_k$ and $\iota_k$ also commute upon restriction to a fibre away from zero.

	In particular, given $s\in H^0(X_m,kL_m)$, we have $t^{-\lambda}\tilde s\in H^0(\cX_m,k\cL_m)$ if and only if $\tilde\tau_k(t^{-\lambda}\tilde s)=t^{-\lambda}\widetilde{\tau_k(s)}\in H^0(\cX,(km\cL)\otimes \tilde\mfb_m^k)$; equivalently, $\cF^\lambda_m H^0(X_m,kL_m)\subseteq \cF^{\lambda} H^0(X,kmL)$, proving Equation \eqref{eq:2} under the above identifications. Since $\iota_1$ and $\tilde \iota_1$ are isomorphisms, this further proves Equation \eqref{eq:1}.

	From Equation \eqref{eq:2} it follows that for each $m$, the subalgebra $R(X_m,L_m)_{\geq \lambda,\cF_m}$ is contained in the subalgebra $R(X,mL)_{\geq m\mi \lambda,\cF}$, and that those algebras coincide in degree one. We may define a rescaled filtration $m\cF_m$ by
	$$m\cF_m^\lambda R_k:=\cF^{m\lambda}_m R_k,$$
	so that we may rephrase the above as saying that the subalgebra $R(X_m,L_m)_{\geq \lambda,m\cF_m}$ is included in $R(X,L)_{\geq \lambda,\cF}$, again with equality in degree one. Hence from Lemma \ref{lem:okounkov} below, we have
	\begin{equation}
		\lim_{m\to\infty} m^{-n}\vol(R(X_m,L_m)_{\geq \lambda,m\cF_m})=\vol(R(X,L)_{\geq \lambda,\cF}).
	\end{equation}

	We write 
	\begin{align*}
		f_m(\lambda)&:=m^{-n}\vol(R(X_m,L_m)_{\geq \lambda,m\cF_m});\\
		f(\lambda)&:=\vol(R(X,L)_{\geq \lambda,\cF}),
	\end{align*}
	and note that they are decreasing functions such that for all $m$ and $\lambda$ we have $0\leq f(\lambda)\leq \vol(X,L)$ and
	\begin{align*}0\leq f_m(\lambda)&\leq m^{-n}\vol(X_m,L_m),\\
	&=m^{-n}\vol(X_m,\mu_m^*(mL)-E_m),\\
	&=\vol(X_m,\mu_m^*L-m\mi E_m)\leq \vol(X,L)
	\end{align*}
	by Fujita approximation. Hence it follows from dominated convergence that the (compactly supported) measures $\nu_m:=-\vol(X,L)\mi m^{-n}f_m(\lambda)$ converge in the sense of distributions to $\nu(\cF):=-\vol(X,L)\mi \frac{d}{d\lambda}f(\lambda)$, thus using Equation \eqref{eq:bchen2},
	\begin{equation}\label{eq:3}
	\int_\bbr \lambda\,d\nu(\cF_m)(\lambda)\to_{m\to\infty} \int_\bbr \lambda\,d\nu(\cF)(\lambda)=m\vol(\cF).
	\end{equation}
	We now unravel the left-hand side of Equation  \eqref{eq:3}. We have that
	\begin{align*}
		\int_\bbr \lambda\,d\nu(\cF_m)(\lambda)&=\frac{1}{\vol(X,L)m^n}\int_\bbr \lambda d(\frac{d}{d\lambda}\vol(R(X_m,L_m)_{\geq \lambda,\cF_m}))(\lambda)\\
		&=\frac{\vol(X_m,L_m)}{\vol(X,L)m^n}\vol(m\cF_m).
	\end{align*}
	On the other hand, it is clear that for each $k$, a basis $\{s_i\}$ of $H^0(X_m,kL_m)$ is adapted to $\cF_m$ if and only if it is adapted to $m\cF_m$. We then have by Equation  \eqref{eq:succmin} that
	\begin{align*}
		\lambda_{i,k}(m\cF_m)&=\sup\{\lambda\in\bbr,\,s_i\in m\cF_m^\lambda R_k\},\\
		&=\sup\{\lambda\in\bbr,\,s_i\in \cF_m^{m\lambda} R_k\},\\
		&=m\mi \sup\{\mu\in\bbr,\,s_i\in \cF_m^{\mu} R_k\}, \\ &=m\mi \lambda_{i,k}(\cF_m).
	\end{align*}
	It then follows from Equation  \eqref{eq:vol1} and Equation  \eqref{eq:vol2} that $\vol(m\cF_m)=m\mi \vol(\cF_m)$, hence we may rewrite Equation  \eqref{eq:3} as
	\begin{equation}\label{eq:4}
		\lim_{m\to\infty}\frac{\vol(X_m,L_m)}{\vol(X,L)m^{n+1}} \vol(\cF_m)=\vol(\cF).
	\end{equation}
	
	It is now known by \cite[Proposition 5.9]{bhj:duistermaat} that, since $\cL_m$ is a semiample test configuration,
	\begin{align*}
		(n+1)\mi\vol(\cF_m)&=\vol(X_m,L_m)\mi (\cL_m^{n+1}),\\
		&=\vol(X_m,L_m)\mi ((\tilde \mu_m^*(m\cL)-\cE_m)^{n+1}),\\
		&=\vol(X_m,L_m)\mi m^{n+1} ((\tilde \mu_m^* \cL-m\mi \cE_m)^{n+1}).
	\end{align*}
	By Fujita approximation we have $((\tilde \mu_m^* \cL-m\mi \cE_m)^{n+1})\to \langle \cL^{n+1}\rangle$, hence by the above and Equation  \eqref{eq:4},
	\begin{align*}
		\vol(\cF)&=\lim_{m\to\infty} ((n+1)\vol(L))\mi ((\tilde \mu_m^* \cL-m\mi \cE_m)^{n+1}), \\ &=\frac{\langle \cL^{n+1}\rangle}{(n+1)\vol(L)}, \\ &=E^\na(\cX,\cL).
	\end{align*}
	This concludes the proof.
\end{proof}

We have used the following lemma due to Boucksom--Jonsson in the proof of Theorem \ref{thm:bchen}. It is stated for $L$ ample \cite[Lemma 4.15]{bj:kstab0}, but the proof, relying on Okounkov bodies \cite{lazmus,kkh,bou:oko}, works without further change in the case where $L$ is big.
\begin{lemma}\label{lem:okounkov}
	Suppose $L$ is a big line bundle on $X$, and let $V_\bullet$ be a graded subalgebra of $R(X,L)$. Suppose for each $m$ sufficiently large and divisible  we are given a subalgebra $W^m_\bullet$ of $R(X,L)$ such that:
	\begin{enumerate}[(i)]
		\item $W^m_1=V_1$;
		\item for each $k>0$, $W^m_k\subseteq V_{mk}$.
	\end{enumerate}
	Then
	$$\lim_{m\to\infty}m^{-n}\vol(W^m_\bullet)=\vol(V_\bullet).$$
\end{lemma}

\subsection{A valuative interpretation of the Monge--Amp\`ere energy} We next give interpretations of the Monge--Amp\`ere energy and its derivative in terms of the associated divisorial valuations, and vice versa, which we expect to be useful in applications.  

To a finite collection $\{E_i,t_i\}$ of divisorial valuations with $t_i \in \R$, we may associate the invariant $$S_L(\{E_i,t_i\})=t_0 + \frac{1}{\vol(L)}\int^{\infty}_{t_0}\vol\left(X,L-\sum_{i=0}^l\max(\lambda-t_i,0)F_i\right)\,d\lambda.$$
We may also find a dominant degeneration $\X$ of $\X$ compatible with the $\{E_i\}$, whose components with Rees valuations corresponding with the $\{E_i\}$ we also denote by $\{E_i\}$, and with the other components denoted by $\{F_j\}$. Hence we may define a line bundle $\L_t = L + \sum_i t_i b_i E_i + c\sum_j F_j$ on $\X$, where $c$ is a sufficiently large  positive constant (which we can choose to depend only on the data of $L$, the $E_i$, and a choice of normalisation for $\min_i t_i$). We will usually omit this $c$ in the notation, as all quantities of interest remain unchanged if we increase $c$ further. This implies that uniform K-stability can be verified either through test configurations, or through arbitrary collections $\{E_i,t_i\}$. We next explain the form of the Monge--Amp\`ere energy, and its derivative, directly through the $\{E_i,t_i\}$.

\begin{lemma}\label{lem:MAS}
The Monge--Amp\`ere energy $E\NA(\X,\L_t)$ satisfies $$E\NA(\X,\L_t) = S_{L}(\{E_i,t_i\})$$
for any $c>c_0$ where $c_0$ can be taken to depend only on $L$, the $E_i$, and the normalisation $\min_i t_i$.
\end{lemma}

\begin{proof}Since $$E\NA(\X,\L_t+\scO_{\pr^1}(c)) = E\NA(\X,\L_t) + c$$ and $$S_{L}(\{E_i,t_i+c\}) = S_{L}(\{E_i,t_i\}) + c,$$  we may assume $\L_t$ is big, in which case this follows from Example \ref{ex:tc} and Theorem \ref{thm:bchen}, noting that if $c$ is sufficiently large, for example if $c \geq \min_i t_i + \max_j \gamma(L,E_j)$ (hence $c \geq \min_i (t_i+\gamma(L,E_i))$ which corresponds to the bound in Theorem \ref{thm:descrfiltr}), then the Rees valuations associated with the $F_j$ will not contribute to the  constraints defining the associated filtration by Theorem \ref{thm:descrfiltr}.
\end{proof}

We further give a divisorial interpretation of $\nabla_HS_L(E_i,t_i),$ extending \cite{bj:kstab2} which treated the case of a single divisorial valuation, using the same idea.

\begin{proposition}\label{prop:diffSL} For fixed $\{E_i,t_i\}$, the function $L \mapsto S_L(\{E_i,t_i\})$ is differentiable, with derivative given by \begin{align*}&\nabla_H S_L(\{E_i,t_i\}) =  \\ &\frac{n}{\langle L\rangle^n}\int_{t_0}^{\infty}  \left\langle L - \sum_i \max\{\lambda-t_i, 0\}E_i\right\rangle^{n-1}\cdot \left(H - \frac{\langle L^{n-1}\rangle\cdot H}{\langle L^n\rangle}\left(L -  \sum_i \max\{\lambda-t_i, 0\}E_i \right)\right) d\lambda.\end{align*}

\end{proposition}

\begin{proof}
First note that each term in the desired equality is linear in $H$, so we may without loss of generality, having decomposed $H=H_1-H_2$, assume that $H$ is ample. We now split the integral defining $  S_L(\{E_i,t_i\})$ into pieces of the form 
\begin{align*} 
   S_L(\{E_i,t_i\}) &=  t_0 + \sum_{i=0}^l\frac{1}{  \langle L^n\rangle}\int_{t_i}^{t_{i+1}} \left\langle L - \sum_{j=0}^i (\lambda-t_j)E_j\right\rangle^n d\lambda\\
    &=t_0 + \sum_{i=0}^l\frac{1}{  \langle L^n\rangle}\int_{t_i}^{t_{i+1}}\left\langle L - \lambda\left(\sum_{j=0}^i E_j\right)  +\sum_{j=0}^i t_jE_j\right\rangle^nd\lambda,
\end{align*}
where we declare $t_l$ to be $\infty$, and we may assume that the $t_i$ are strictly increasing.

Examining each integral separately, there are three possible cases. If $ L - \sum_{j=0}^i (t_i-t_j)E_j$ is big, we may simply differentiate under the integral by differentiability of volume on the big cone to find
\begin{equation}\label{desired-equality}\frac{d}{d\varepsilon}\bigg|_{\varepsilon=0}\int_{t_i}^{t_{i+1}}\left\langle L+\varepsilon H - \lambda\left(\sum_{j=0}^i E_j\right)  +\sum_{j=0}^i t_jE_j\right\rangle^nd\lambda=n\int_{t_i}^{t_{i+1}}\left\langle L - \lambda\left(\sum_{j=0}^i E_j\right)\right\rangle^{n-1}\cdot H\,d\lambda.\end{equation}

For  the first $i$ such that $ L - \sum_{j=0}^i (t_i-t_j)E_j$ is not big, the volume function is not integrable throughout the corresponding interval. Nevertheless, changing variables $\mu=\lambda-t_i$, we find ourselves in the situation of \cite[Theorem 2.18]{bj:kstab2}, which implies the analogue of Equation \eqref{desired-equality}.     Finally, for any further $j$, the associated integral on $[j,j+1]$ will be zero.     The result in the proposition then follows from the quotient rule. \end{proof}

 In particular  $$\nabla_L S_L(\{E_i,t_i\}) = \frac{n}{\langle L\rangle^n}\int_{t_0}^{\infty}  \left\langle L - \sum_i \max\{\lambda-t_i, 0\}E_i\right\rangle^{n-1}\cdot\left(\sum_i \max\{\lambda-t_i, 0\}E_i\right) d\lambda,$$ which implies the norm is nonnegative.

\subsection{Directional differentiability of the Legendre transform of $S_L$.}

We consider a divisorial measure $\mu = \{E_i,\xi_i\}$ and a $\mu$-compatible dominant degeneration $\X$ of $X$ with  $\cX_0=\sum_i b_i E_i$. We denote by $S\subset X^\div$ the support of $\mu$. We are now able to relate the divisorial expression of $\|\mu\|_L$ with quantities associated to test configurations.

\begin{proposition}\label{prop:normE}
The norm $\|\mu\|_L$ of $\mu$ satisfies
	$$\|\mu\|_L=\sup_{t\in \bbr^\ell}\left\{ E^\na(\cX,\cL_t) -\langle \xi,t\rangle\right\}$$
	where $\cL_t:=\pi^*L+\sum_i t_i b_i E_i$.
\end{proposition}
\begin{proof}
	By Theorem \ref{thm:bchen} we have $E^\na(\cX,\cL_t)=\vol(\cF_{\cL_t})$ where $\cF_{\cL_t}$ is the filtration associated with the test configuration $(\cX,\cL_t)$ in the sense of Example \ref{ex:tc}, which are then divisorial filtrations supported on $\Rees(\cX)$. By Theorem \ref{thm:descrfiltr} and Proposition \ref{prop:sup}, the supremum on the right-hand side therefore equals the supremum defining the left-hand side, concluding the proof.
\end{proof}

We will now use a theorem that is standard in optimisation theory, but may not be well-known to readers in the field, known as Danskin's theorem \cite{danskin}. We give a proof of the exact version of the statement we will use in this article for the convenience of the reader, which very closely follows the proof of \cite[Theorem 1.29]{book:guler} (see also \cite[Theorem 10.31]{book:rockafellarvariational} for a more general result).
\begin{theorem}[Danskin's theorem]\label{thm:danskin}
Let $f:U\times K\to\bbr$ be a function with $U$ open in $\bbr^m$ and $K$ a compact space, such that:
\begin{enumerate}[(i)]
    \item $f$ is continuous;
    \item for all $(x,y)\in U\times K$, and all $h\in \bbr^m$, the directional derivative $\nabla_h f(x,y)=\lim_{t\to 0} t\mi (f(x+th,y)-f(x,y))$ exists;
    \item for all $h\in\bbr^m$, $(x,y)\mapsto \nabla_h f(x,y)$ is continuous.
\end{enumerate}
Then $g(x):=\sup_{y\in K} f(x,y)$ is directionally differentiable in all directions, with
$$\nabla_h^+ g(x)=\sup_{y\in \mathrm{argmax} f(x,\cdot)} \nabla_h f(x,y).$$
\end{theorem}

Here $ \mathrm{argmax} f(x,\cdot)$ denotes the set of $y$ with $f(x,y) = \max_{z\in K}f(x,z)$, and $\nabla^+$ denotes the usual (right) directional derivative.
\begin{proof}
    Let $h\in \bbr^m$, and fix $x\in U$. Let $y\in \mathrm{argmax} f(x,\cdot)$, and for all $k>0$, let $y_k\in \mathrm{argmax} f(x+k\mi h,\cdot)$, so that $f(x,y)=g(x)$, $f(x+k\mi h,y_k)=g(x+k\mi h)$. Then 
    \begin{align*}
        k\mi(g(x+k\mi h)-g(x))&=k\mi(f(x+k\mi h,y_k)-f(x,y))\\
        &\geq k\mi(f(x+k\mi h,y)-f(x,y)).
    \end{align*}
    By the mean value theorem applied to $t\mapsto f(x+th,y)$, which is differentiable by assumption, this in fact equals $\nabla_h f(x+t_k h,y)$ for some $0< t_k <k\mi$. Taking the limit in $k$ using that directional derivatives are continuous, and since the above holds for an arbitrary maximiser $y$, we find
    $$\liminf_{k\to\infty} k\mi(g(x+k\mi h)-g(x))\geq \sup_{y\in \mathrm{argmax} f(x,\cdot)} \nabla_h f(x,y).$$
    Now, by compactness, up to taking a subsequence we may assume that $y_k$ converges to some $y'\in K$. We first claim that $g(x)=f(x,y')$, i.e.\ $y'\in \mathrm{argmax} f(x,\cdot)$ as well. Indeed, let $y\in \mathrm{argmax}f(x,\cdot)$, then 
    $$f(x,y')=\lim_{k\to\infty}f(x+k\mi h,y_k)\geq \lim_{k\to\infty} f(x+k\mi h,y)=f(x,y)$$
    but $f(x,y)\geq f(x,y')$ by definition of $y$. Thus,
    \begin{align*}
        k\mi(g(x+k\mi h)-g(x))&=k\mi(f(x+k\mi h,y_k)-f(x,y'))\\
        &\leq k\mi(f(x+k\mi h,y_k)-f(x,y_k))\\
        &=\nabla_h f(x+s_k h,y_k)
    \end{align*}
    for some $0<s_k<k\mi$ by the mean value theorem again. Hence by continuity of directional derivatives we again have that
    \begin{align*}
        \limsup_{k\to\infty} k\mi(g(x+k\mi h)-g(x))&\leq \nabla_h f(x,y')\leq \max_{y\in \mathrm{argmax}f(x,\cdot)} \nabla_h f(x,y),
    \end{align*}
    concluding the proof.
\end{proof}

We apply this result to understand the derivative of the norm.

\begin{theorem}\label{thm:diff}
	The function $L\mapsto \|\mu\|_L$ is directionally differentiable in all directions at all points $L\in Big(X)$; furthermore, for any $\cX$ which is $\mu$-compatible, then for any $\R$-Cartier divisor $D$ on $X$,
	$$\langle L^n\rangle \nabla^+_{D} \|\mu\|_L=\sup_{\MA(\cX,\cL)=\mu}\left\{\langle\cL^n\rangle\cdot \pi^*D-n(\langle L^{n-1}\rangle \cdot D)E^\na(\cX,\cL)\right\}.$$
\end{theorem}
\begin{proof}
We take $U = \Big(X)$ the big cone and $K$ a compact subset of $\R^{|S|}$, and define $$f(L,t):=-\langle \xi,t\rangle + E^\na\left(\cX,L+\sum_i t_i b_i E_i\right).$$ By Proposition \ref{prop:normE} we have $\|\mu\|_L=\sup_{t\in \bbr^{|S|}} f(L,t)$. By Theorem \ref{thm:danskin}, assuming we can show that:
	\begin{enumerate}[(i)]
		\item there exists an open neighbourhood $U\subseteq \Big(X)$ of $L$ and $K\subset \bbr^{|S|}$ compact such that for all $L'\in U$,
		$$\sup_{t\in \bbr^{|I|}}f(L',t)=\sup_{t\in K}f(L',t);$$
		\item $f$ is continuous on $U\times K$;
		\item for all $D,L,t$, the directional derivative $\nabla_D f(L,t)$ exists;
		\item for all $D$, $\nabla_D f(L,t)$ is continuous in $(L,t)$ on $U\times K$,
	\end{enumerate}
	then $\norm[\mu]_L$ is directionally differentiable at $L$, and further
	\begin{equation}
		\nabla_D \|\mu\|_L=\sup_{t\in \mathrm{argmax}(f(L,\cdot))} \nabla_D f(L,t).
	\end{equation}
	By Theorem \ref{thm:variationalE} $(i)$, the values $t\in\R^{|S|}$ with $t\in\mathrm{argmax}(f(L,\cdot))$ correspond to solutions  the Calabi--Yau theorem with respect to the divisorial measure $\mu$,  so this will prove the theorem provided we are able to prove that $(i)$-$(iv)$ hold.
	
	Choose an open neighbourhood $U\subset \Big(X)$ of $L$ of radius $\varepsilon>0$. We first note that, since $f(L,\cdot)$ is invariant under the translation action of $(\bbr,+)$ on $\bbr^{|S|}$ as a function of $t$, we can normalise to $\min_i t_i=c$ for $c$ sufficiently large  so that for each such $t$, for all  $L'\in U$ the line bundle $L'+\sum_i t_i b_i E_i$ is big. Now, by Theorem \ref{thm:descrfiltr} as in Corollary \ref{coro:compactness} (and Example \ref{ex:tc}), for a fixed $L'\in U$, $f(L',\cdot)$ is maximised in the set of vectors $t$ such that, for each component $t_i$,
$$c\leq t_i \leq c+\max_{i\in I} \gamma(L',E_i)+1.$$
Note furthermore that each $L\mapsto \gamma(L,E_i)$ is concave, hence locally Lipschitz, on $\Big(X)$; in particular, the maximum of finitely many such functions is also locally Lipschitz. Let $C$ thus be a local Lipschitz constant for $L\mapsto \max_{i\in I} \gamma(L,E_i)$ on $U$, so that for an Euclidean norm on $\Big(X)$ and any $L'\in U$,
$$|\max_{i\in I} \gamma(L,E_i) - \max_{i\in I} \gamma(L',E_i)|\leq C\norm[L-L']\leq C\varepsilon.$$
In particular, for all $L'\in U$, $f(L',\cdot)$ is also maximised in the compact set
\begin{equation}\label{danskin-app}c\leq t_i \leq c+C\varepsilon+\max_{i\in I} \gamma(L,E_i)+1\end{equation}
which we denote by $K$. This choice of $(U,K)$ gives (i). The continuity property $(ii)$ is then an immediate consequence of continuity of the volume on the big cone.
	
	Likewise, it follows from differentiating $E^\na$ in the direction of $D$ that $(iii)$ holds, with
	$$\nabla_D f(L,t)=\langle L^n\rangle\mi\left\langle(\pi^*L+\sum_i t_i b_i E_i)^n\right\rangle\cdot \pi^*D-\frac{n\langle L^{n-1}\rangle \cdot D}{\langle L^n\rangle}E^\na\left(\cX,\pi^*L+\sum_i t_i b_i E_i\right).$$
	It is then clear that $(iv)$ holds, since positive intersection products are continuous on the big cones of $X$ and $\cX$ \cite[Proposition 2.9]{bfj:teissier}. This concludes the proof.

\end{proof}

\begin{remark}\label{rem:nablamin}What we will use from this result is that 
\begin{align*}\nabla^-_{K_X}\|\mu\|_L&=-\nabla^+_{-K_X}\|\mu\|_L,\\
	&=\inf_{\MA(\cX,\cL)=\mu}\frac{\langle\cL^n\rangle\cdot \pi^*K_X}{\langle L^n\rangle}-\left(\frac{n\langle L^{n-1}\rangle\cdot K_X}{\langle L^n\rangle} \right)E^\na(\cX,\cL).
\end{align*}
\end{remark}

\subsection{Divisorial stability and K-stability}\label{subsect:mainsect}

We are now equipped to relate divisorial stability and uniform K-stability.

\begin{lemma}\label{lem:entropy}
Let $(\cX,\cL)$ be a test configuration. Then
	$$\vol(L)\mi\sum_i b_i A_X(E_i) \langle \L\rangle^n\cdot E_i  =H^\na(\cX,\cL).$$
\end{lemma} 
\begin{proof}

The claim follows from a calculation of $K_{\X/X\times\pr^1}$ in terms of the discrepancies $A_X(E_i)$, which is due to Odaka \cite[Theorem 4.1]{odaka-ross-thomas} (see also \cite[Corollary 4.12]{bhj:duistermaat}).
\end{proof}

\begin{remark}
If one views $\MA(\X,\L)$ as a measure on $X^\div$, this can be phrased as the statement that the integral of the log discrepancy function against  $\MA(\X,\L)$ agrees with $H^\na(\cX,\cL)$.
\end{remark}

We will rely on the following result to relate the norms, which proves genuine differentiability of $\|\mu\|_L$ in the direction of $L$ itself. Let $\mu = \{E_i,t_i\}$ be a divisorial measure, with $S = \{E_i\}_i \subset X^{\div}$. 

\begin{proposition}\label{prop:derivativeUL}
We have 
 $$\nabla_L^- \|\mu\|_L=\nabla^+_L \|\mu\|_L=\|\mu\|_L.$$ 
\end{proposition}
\begin{proof}
Recall $$ \|\mu\|_{L} = \sup_{t \in \R^{|S|}} (-\langle \xi,t\rangle + S_{L}(t)),$$ and recall we may assume $\min_i t_i=0$ in this supremum. It follows that  we may write
$$S_{L}(t) = \frac{1}{\vol(L)}\int_0^{\infty}\vol\left(L - \sum_{i=0}^l \max(\lambda - t_i,0)E_i\right)d\lambda.$$ 
For $s\in \bbr$ sufficiently small, homogeneity of the volume implies
$$S_{(1+s)L}(t) =  \frac{1}{\vol(L)}\int_0^{\infty}\vol\left(L - \sum_{i=0}^l \max\left(\frac{\lambda - t_i}{1+s},0\right)E_i\right)d\lambda,$$ so that  by a change of variables we have $$S_{(1+s)L}((1+s)t)=(1+s)S_{L}(t).$$

The norm of $\mu$ with respect to $(1+s)L$ similarly satisfies, where we may assume $\min_i t_i=0$
\begin{align*}
\|\mu\|_{(1+s)L} &= \sup_{(1+s)t \in \R^{|S|}} (-\langle \xi,(1+s)t\rangle + S_{(1+s)L}((1+s)t)),\\
 &= (1+s)\sup_{(1+s)t \in \R^{|S|}} \left(-\langle \xi,t\rangle + S_{L}(t)\right).
\end{align*} It follows that the norm also satisfies the homogeneity property
$$\|\mu\|_{(1+s)L} = (1+s)\|\mu\|_L,$$ which implies the result.\end{proof}

\begin{corollary}\label{coro:norm}
	If $\mu=\MA(\X,\L)$, then $\norm[\mu]_L=\norm[(\X,\L)]_L$.
\end{corollary}
\begin{proof}
	By Proposition \ref{prop:derivativeUL}, the left and right directional derivatives in the direction of $L$ of $\norm[\mu]_L$ coincide, thus by Theorem \ref{thm:diff} and Remark \ref{rem:nablamin} applied to $L$ and $-L$ we have that 
	$$\nabla_L \|\mu\|_L=\nabla_L E^\na(\cX,\cL)$$
	for any $(\cX,\cL)$ satisfying $\MA(\cX,\cL)=\mu$. Since $\norm[(\cX,\cL)]_L=\nabla_L E^\na(\cX,\cL)$ and $\norm[\mu]_L=\|\mu\|_L$, the result then follows from Proposition \ref{prop:derivativeUL}.
\end{proof}

The following is the main result of the section.

\begin{theorem}\label{thm:divequalsK}
$(X,L)$ is uniformly K-stable if and only if it is divisorially stable.
\end{theorem}
\begin{proof}

Suppose first that $(X,L)$ is divisorially stable, and let $(\X,\L)$ be a test configuration, so that $\MA(\X,\L)$ is a divisorial measure which we denote $\mu = \{E_i,\xi_i\}$. Lemma \ref{lem:entropy} implies that $$H^\na(\cX,\cL)=\sum_i \xi_i A_X(\nu_i),$$ while Corollary \ref{coro:norm}  implies $\norm[\mu]_L=\norm[(\X,\L)]_L$. The inequality $$ \nabla_{K_X} E^\na(\cX,\cL) \geq \nabla^{-}_{K_X}\|\mu\|_L,$$ which follows from Theorem \ref{thm:diff} and Remark \ref{rem:nablamin}, then implies that $$M\NA(\X,\L) \geq \epsilon \norm[(\X,\L)]_L,$$ in turn implying that $(X,L)$ is uniformly K-stable.

	Conversely, suppose $(X,L)$ is uniformly K-stable, and let $\mu$ be a divisorial measure. By Theorem \ref{thm:main} we can find $(\cX,\cL)$ a test configuration such that $\MA(\cX,\cL)=\mu$. By the above arguments, the entropy and norm terms are again equal, and Remark \ref{rem:nablamin} gives $$\nabla^{-}_{K_X}\|\mu\|_L= \inf_{\MA(\cX,\cL)=\mu}\frac{\langle\cL^n\rangle\cdot \pi^*K_X}{\langle L^n\rangle}-\frac{n\langle L^{n-1}\rangle \cdot K_X}{\langle L^n\rangle}E^\na(\cX,\cL),$$ so taking a sequence of test configurations minimising this quantity proves that $(X,L)$ is divisorially stable.
	\end{proof}
	
\begin{remark}
The proof in fact implies that the modulus of stability involved in uniform K-stability agrees with that for divisorial stability.
\end{remark}

\subsection{K-stability with big anticanonical class} \label{sec:fano-case}

Notions of stability were previously introduced by the authors \cite{derreb:1} and Darvas--Zhang  \cite{dar:zhangbigke}  for a klt variety $X$ with big anticanonical class $-K_X$. The stability condition of  \cite{derreb:1} was \emph{uniform Ding stability}, which involves test configurations for birational models of $X$ and associated Ding invariants, whereas the approach of \cite{dar:zhangbigke} involved the $\delta$-invariant associated to divisorial valuations $E \subset Y \to X$ on birational models of $X$, defined by $$\delta(X):=\inf_{E}\frac{A_X(E)}{S_{-K_X}(E)};$$ stability in their sense requires $\delta(X)>1$. From  \cite{dar:zhangbigke, derreb:1}, these conditions are further equivalent to the existence of a singular K\"ahler--Einstein metric on $X$, provided $X$ is smooth.  These two notions of stability were subsequently proven to be equivalent by Xu  \cite{xu:big}, who further showed that these stability conditions force the anticanonical ring to be finitely generated, and that stability is equivalent to log K-stability of an associated log Fano pair. We complete this picture, by showing that these conditions are further equivalent to K-stability of $(X,-K_X)$, extending \cite[Theorem 4.10]{bj:kstab2} which treated the ample case (and for which we assume $X$ is smooth).

\begin{theorem}\label{thm:fano}
	Assume $-K_X$ is big. Then $(X,-K_X)$ is K-stable if and only if $\delta(X)>1$.
\end{theorem}

From the above discussion, it follows that K-stability is also equivalent to uniform Ding stability, while Theorem \ref{thm:divequalsK} further characterises these conditions through divisorial stability.

\begin{proof}

We prove that divisorial stability is equivalent to the condition that $\delta(X)>1$. We begin by noting that Proposition \ref{prop:derivativeUL} simplifies the expression of $\beta(\mu)$, reducing to the case $t_0=0$: $$\beta(\mu) =\sum_i \xi_i A_X(E_i)-\|\mu\|.$$

The condition that $\delta(X)>1$ is equivalent to the existence of $\varepsilon>0$ such that for all divisorial valuations $E$,  
$$A_X(E)-S_L(E)\geq \varepsilon S_L(E),$$ from which it is clear  that divisorial stability implies $\delta(X)>1$, by taking $\mu=\{(E,1)\}$.

Conversely, suppose $\delta(X)>1$, and let $\mu =\{E_i,\xi_i\}$ now be a divisorial measure with $\{E_i\} = S\subset X^{\div}$. Then we first note that we may see $\|\cdot \|_{-K_X}$ as a function of $\xi\in \bbr^{|S|}$ with basis identified with $\{E_i\}_i$, so
$$\|\mu\|_{-K_X}=\sup_{t\in \R^{|S|}}\{S_{L,S}(t) -\langle \xi,t\rangle\}.$$
This is a supremum of affine functions in $\xi$, hence is convex in $\xi$. Therefore
\begin{align*}\beta(\mu) &= \sum_i \xi_i A_X(E_i) - \|\xi\|_{-K_X}, \\ &\geq  \sum_i \xi_i A_X(E_i)  - \sum_i \xi_i \|E_i\|_{-K_X}
\end{align*}
with $E_i$ viewed as a basis vector in $\bbr^{|S|}$. Under the above identifications, from Remark \ref{rem:expectedorder} we have $ \|E_i\|_{-K_X}=S_{-K_X}(E_i)$, thus
\begin{align*}\beta(\mu) &\geq  \sum_i \xi_i \beta(E_i), \\ &\geq \epsilon \sum_i \xi_i \|E_i\|_{-K_X}, \\&\geq\varepsilon \|\mu\|_{-K_X}. \end{align*} The result follows.\end{proof}

\section{The birational behaviour of K-stability}

In this section, we aim to prove Theorem \ref{intro:birational}.

\subsection{Pullback invariance}\label{sec:pullback}

We let $\pi:Y\to X$ be a birational morphism between smooth projective varieties, let $L_X$ be a big line bundle on $X$, and denote $L_Y = \pi^*L_X$. We are interested in the behaviour of uniform K-stability under such morphisms, and first prove the following.

\begin{theorem}\label{thm:pullbackinvariance}
    Let $\pi:Y\to X$ be a birational morphism between smooth projective varieties, and $L_X$ be a big line bundle on $X$. Let $L_Y:=\pi^*L_X$. Then, $(X,L_X)$ is uniformly K-stable if and only if $(Y,L_Y)$ is uniformly K-stable. 
\end{theorem} 

We begin with the following technical result, which proves that equality of filtrations implies equality of numerical invariants.

\begin{lemma}\label{lem:comparison}If two test configurations $(\Y,\L_Y)$ and $(\Y',\L'_Y)$ induce the same filtration on $R(Y,L_Y)$, then
$$M^\na(\Y,\L_Y)=M^\na(\Y',\L'_Y),\quad \norm[(\Y,\L_Y)]=\norm[(\Y',\L'_Y)].$$
\end{lemma}
\begin{proof}

Since the filtration induced by a test configuration, and the Monge--Ampère measure, are invariant under pullback of test configurations, we may assume $\L_Y$ and $\L'_Y$ are determined on the same degeneration $\cY$ with $\cY_0=\sum_i b_i E_i$, denoting by $\nu_i$ the Rees valuation of each $E_i$. We may then write $\cL_Y=L_Y+\sum_i t_i b_i E_i$, $\cL'_Y=L_Y+\sum_i s_i b_i E_i$, and thus the associated filtrations are
 $$\cF^\lambda H^0(Y,kL_Y)=\{s\in H^0(Y,kL_Y),\,\min_i\{\ord_{E_i}(s) + kt_i\}\geq \lambda\},$$
  $$\cF'{}^\lambda H^0(Y,kL_Y)=\{s\in H^0(Y,kL_Y),\,\min_i\{\ord_{E_i}(s) + ks_i\} \lambda\,\forall i\}.$$
  Without loss of generality, we can assume $\min_i t_i = \min_i s_i=c$ for some $c$ succifiently large , which corresponds to adding a multiple $c\cO(1)$, so that $\cL_Y+c\cO(1)$ and $\cL'_Y+c\cO(1)$ are big on $\cY$.
  
  Let $\gamma_i:=\gamma(L_Y,E_i)$. Define, for each $i$, a real number $u_i$ by setting it to be
  	$$u_i:=\min\{s_i,t_i,\min_i\{\gamma_i+t_i\},\min_i\{\gamma_i+s_i\}\}.$$
  	 We first claim that the filtration $\cG$ induced by $\cM_Y:=L_Y+\sum_i u_i b_i E_i$ equals the filtrations induced by $\cL_Y$ and $\cL'_Y$. Indeed, by Theorem \ref{thm:descrfiltr}(i,ii), they induce the same subalgebras for $\lambda<c$ and $\lambda>\min\{\min_i \{\gamma_i+t_i\},\min_i\{\gamma_i+s_i\}\}$. Now, assume $\lambda$ is inbetween those bounds, and that $\min_i\{s_i,t_i\}<\min\{\min_i \{\gamma_i+t_i\},\min_i\{\gamma_i+s_i\}\}$, so that $u_i=\min\{s_i,t_i\}$. That $s\in \cF^\lambda H^0(Y,kL_Y)=\cF'{}^\lambda H^0(Y,kL_Y)$ implies for all $i$ that $\ord_{E_i}(s)+kt_i\geq \lambda$ and $\ord_{E_i}(s)+ks_i\geq \lambda$, hence $\ord_{E_i}(s)+ku_i\geq \lambda$; in particular, $\cF=\cF'\subseteq\cG$ as filtrations. Conversely, if $\ord_{E_i}(s)+ku_i\geq \lambda$ then $\ord_{E_i}(s)+kt_i\geq \lambda$, since $u_i\leq t_i$, so that $\cG\subseteq \cF=\cF'$.
  	 
  	 We have thus shown that there exists a test configuration $\cM_Y$ with $\cL_Y-\cM_Y\geq 0$ which induces the same filtration as $\cL_Y$. By monotonicity of positive intersection products, this implies
  	 $$\MA(\cY,\cM_Y)\leq \MA(\cY,\cL_Y),$$
  	 but both are probability measures, which implies they must be equal. We may apply the same argument to $\cL_Y'$, which finally implies
  	 $$\MA(\cY,\cL_Y)=\MA(\cY,\cL'_Y).$$
  	 By Lemma \ref{lem:entropy}, this implies $H^\na(\Y,\cL_Y)=H^\na(\Y,\cL'_Y)$.
  	 
  	 Equality of the norm terms is clear, since $S_L(\{E_i,t_i\})=S_L(\{E_i,s_i\})$, and we now argue similarly for the remaining energy term in $M^\na$. Let $|\varepsilon|$ small, and consider the filtrations $\cF_\varepsilon$ and $\cG_\varepsilon$ induced by $\{E_i,t_i\}$ and $\{E_i,u_i\}$ on $R(Y,L_Y+\varepsilon H)$. It is no longer the case that both filtrations agree on this modified section ring, but since $u_i\leq t_i$ we still have that $\cG_\varepsilon\subseteq \cF_\varepsilon$, so that in particular
  	 $$S_{L_Y+\varepsilon H}(\{E_i,t_i\})\geq S_{L_Y+\varepsilon H}(\{t_i,u_i\}).$$
	Seen as functions of $\varepsilon$ sufficiently small, both are $C^1$, comparable, and agree at $\varepsilon=0$, hence their derivatives at zero agree, giving
	$$\nabla_H S_{L_Y}(\{E_i,t_i\})=\nabla_H S_{L_Y} (\{E_i,u_i\}).$$
	By Lemma \ref{lem:MAS} (wherein we may in fact pick the constant $c$ so that the lemma applies to all $L'$ in a sufficiently small neighbourhood of $L_Y$ within the big cone), this implies
	$$\nabla_H E^\na(\Y,\L_Y)=\nabla_H E^\na(\Y,\L_Y'),$$
  	thereby concluding the proof.

\end{proof}

Before beginning our proof of Theorem \ref{thm:pullbackinvariance}, we note that we may assume that $\pi$ is a blowup. Indeed, we may first find a third $Z$ with compatible morphisms $\pi_Y: Z\to Y$ and $\pi_X: Z \to X$ such that the exceptional loci are of codimension one, for example by taking the normalised blowup of the exceptional locus of $\pi: Y\to X$. Being divisorial contractions, $\pi_Y$ and $\pi_Z$ are blowups, and it is clear that the result for $\pi_Y$ and $\pi_X$ implies the result for $\pi$. We thus assume that $\pi$ is a blowup of an ideal sheaf $\mfa$, with (possibly reducible and nonreduced, but effective) exceptional divisor $E$. We begin by proving the following.

\begin{proposition}\label{prop:pullback1}

Suppose $(\X,\L_X)$ is a test configuration for $(X,L_X)$, with $\X$ smooth, and $\Y$ a smooth degeneration of $Y$ with a $\bbc^*$-equivariant morphism $\tilde \pi:\Y\to\X$ commuting with $\pi$. Then, there exists a line bundle $\L_Y$ on $\Y$ such that:
\begin{enumerate}[(i)]
	\item $(\Y,\L_Y)$ is a test configuration for $(Y,L_Y)=(Y,\pi^*L_X)$;
	\item $M\NA(\X,\L_X) = M\NA(\Y,\L_Y)$;
	\item $\|(\X,\L_X)\| = \|(\Y,\L_Y)\|$;
	\item $(\Y,\L_Y)$ and $(\X,\L_X)$ induce the same filtrations on the isomorphic section rings $(X,L_X)\simeq (Y,L_Y)$.
\end{enumerate}
\end{proposition}

\begin{proof}
Let $\L_Y:=\tilde \pi^*\L_X$. We claim that $M\NA(\X,\L_X) = M\NA(\Y,\L_Y)$. Indeed, as Weil divisors, $K_{\Y} = K_{\X} + \E$, where $\E$ is $\pi$-exceptional (but not necessarily effective), and $K_Y$ and $K_X$ compare similarly. As such, the argument of Proposition \ref{lem:can-assume-snc} (which is a consequence of the birational behaviour of positive intersection products, through Equations \eqref{tess-1} and \eqref{tess-2}) applies to this situation, proving the claim. The same argument implies $\|(\X,\L_X)\| = \|(\Y,\L_Y)\|$. The claim about filtrations is clear, by the definition of the filtration associated to a test configuration given in Example \ref{ex:tc}.
\end{proof}

This implies that uniform K-stability of $(Y,L_Y)$ implies that of $(X,L_X)$, and what remains is to prove the converse.

\begin{proof}[Proof of Theorem \ref{thm:pullbackinvariance}.]

Suppose $(Y,L_Y)$ is uniformly K-stable. Let $(\X,\L_X)$ be a test configuration for $(X,L_X)$. We now produce a degeneration $\Y$ of $Y$ satisfying the hypothesis of Proposition \ref{prop:pullback1}, analogously to Stoppa for blowups at points \cite{stoppa:1,stoppa:2}.

Without loss of generality, we may assume $\X$ is smooth. Realising $\pi: Y \to X$ as the blowup of an ideal sheaf $\mfa$, we take the normalised blowup of the $\C^*$-closure of $\mfa$ in $\X$, denoted $\Y$. Note $\Y$ naturally admits a $\C^*$-action, being the blowup of a $\C^*$-invariant ideal, while as $\Y$ is normal, the surjective morphism $\Y \to\C$ is automatically flat, and it is an obvious consequence of the blowup construction that $(\Y,\L_Y)$ restricts to $(Y,L_Y)$ away from zero. We may then replace $\Y$ by an equivariant resolution of singularities, which satisfies the same properties. Proposition \ref{prop:pullback1} then shows $(X,L_X)$ is uniformly K-stable.

Suppose now $(X,L_X)$ is uniformly K-stable. Given a test configuration $(\Y,\L_Y)$ for $(Y,L_Y)$, we can consider its associated divisorial filtration of the form  $\cF_{(E_i,t_i)}$. We may produce a degeneration $\X$ of $X$ such that the $E_i$ are realised as components of $\X_0$, which we write as $E_{i,X}$ respectively, and without loss of generality we may assume there exists a morphism $\pi:\cY\to \cX$. This produces a test configuration $(\X,\L_X)$, where   
\begin{equation}\label{eq:tc1} 
\L_X = \sum_i t_i E_{i,X}+c\sum_j F_j,
\end{equation} 
where the $F_j$ are the remaining components of the central fibre of $\X$, and $c$ is a sufficiently large constant so that the $F_j$ do not contribute to the filtration induced by $\L_X$, as in Theorem \ref{thm:descrfiltr}. In particular, the filtration induced by $\L_X$ and $\L_Y$ on the (isomorphic) section rings of $L_X$ and $L_Y$ coincide; they also equal the filtration induced by $\pi^*\cL_X$ by Proposition \ref{prop:pullback1}.

Assume then that $(X,L_X)$ is uniformly K-stable, so that
\begin{equation}\label{eq:kstab}
M\NA(\X,\L_X) \geq \varepsilon \|(\X,\L_X)\|.
\end{equation}
By Proposition \ref{prop:pullback1}, we have
\begin{equation*}
M\NA(\X,\L_X)=M\NA(\Y,\pi^*\L_X),\,\quad \|(\X,\L_X)\|=\|(\Y,\pi^*\L_X)\|.
\end{equation*}
By the discussion above on the induced filtrations, from Lemma \ref{lem:comparison}, 
\begin{equation*}
M\NA(\Y,\L_Y)=M\NA(\Y,\pi^*\L_X),\,\quad \|(\Y,\L_Y)\|=\|(\Y,\pi^*\L_X)\|.
\end{equation*}
Thus $(Y,L_Y)$ is uniformly K-stable, concluding the proof of $(ii)$.\end{proof}

\subsection{Pushforward invariance.}

We next turn to the proof of Theorem \ref{intro:birational}, which we recall the statement of:

\begin{theorem}\label{thm:pushforwards}
    Let $\pi:Y\dashrightarrow X$ be a birational contraction between smooth projective varieties. Let $L_Y$ be a big line bundle on $Y$ and $L_X$ be a big line bundle on $X$ such that, on the pullback to a resolution of indeterminacy $Z$, $L_Y = L_X+E$ with $E$ an exceptional effective $\bbr$-Cartier divisor. Then, $(X,L_X)$ is uniformly K-stable if and only if $(Y,L_Y)$ is uniformly K-stable. 
\end{theorem} 

Although this result implies Theorem \ref{thm:pullbackinvariance}, we actually use that result to prove Theorem \ref{thm:pushforwards}, since in the statement the map $\pi$ is only assumed to be a rational map. In particular, by Theorem \ref{thm:pullbackinvariance}, we may assume that $Y=Z$, with $L_Y = L_X + E$ for $E$ exceptional and effective, so that $L_X=\pi_*L_Y$;  we may further assume that $\pi$ is a blowup.

We now give the necessary ingredients in the proof of Theorem \ref{thm:pushforwards}.

\begin{proposition}\label{prop:pullbacks}

Suppose $(\X,\L_X)$ is a test configuration for $(X,L_X)$, with $\X$ smooth, and $\Y$ a smooth degeneration of $Y$ with a $\bbc^*$-equivariant morphism $\tilde \pi:\Y\to\X$ commuting with $\pi$. Then, there exists a line bundle $\L_Y$ on $\Y$ such that:
\begin{enumerate}[(i)]
	\item $(\Y,\L_Y)$ is a test configuration for $(Y,L_Y)=(Y,\pi^*L_X+E)$;
	\item $M\NA(\X,\L_X) = M\NA(\Y,\L_Y)$;
	\item $\|(\X,\L_X)\| = \|(\Y,\L_Y)\|$;
	\item $(\Y,\L_Y)$ and $(\X,\L_X)$ induce the same filtrations on the isomorphic section rings $(X,L_X)\simeq (Y,L_Y)$.
\end{enumerate}
\end{proposition}
\begin{proof}
We set
$$\cL_Y:=\tilde\pi^*\cL_X+E',$$
where $E'$ is the strict transform of $E\times\bbp^1\subset Y\times\bbp^1$, which is $\R$-Cartier since $\cY$ is smooth. Since $E$ is $\R$-Cartier, $\pi$-exceptional and effective, we have 
$$H^0(Y,m(\pi^*L_X + E)) = H^0(X,mL_X)$$
by the projection formula, so $\langle L_X^n\rangle = \langle L_Y^n\rangle.$ Similarly for $F \subset Y$ a second $\pi$-exceptional, effective divisor $$H^0(Y,m(\pi^*L_X + E + tF)) =H^0(X,mL_X),$$ again by the projection formula. This implies, using differentiability of volume, that $ \langle L_X^{n-1}\rangle\cdot F = 0;$ this also follows from Equation \eqref{tess-2}. It follows that  $$\langle L_X^{n-1}\rangle\cdot K_X =  \langle L_Y^{n-1}\rangle \cdot K_Y,$$ since the difference of $K_Y$ and $\pi^*K_X$ is exceptional. The same arguments imply $\langle \L_Y^{n+1}\rangle = \langle \L_X^{n+1}\rangle$ and $\langle \L_{X}^{n}\rangle\cdot K_{\X} =  \langle \L_Y^{n}\rangle\cdot K_{\Y}$, which is enough to imply the result. 

For the associated filtrations, note for $k$ sufficiently large and divisible that $s\in \cF^\lambda_{(\X,\L_X)}H^0(X,kL_X)$ if and only if $t^{-\lambda}\tilde s$ is a section of $k\cL_X$. This is the case if and only if $t^{-\lambda}(\pi^*\tilde s)\otimes s_{E'}$ is a section of $k\cL_Y$, but $(\pi^*\tilde s)\otimes s_{E'}$ is the $\bbc^*$-closure of $(\pi^*s)\otimes s_E$, meaning this is equivalent to $\pi^*s\otimes s_E\in \cF^\lambda_{(\Y,\L_Y)}R(Y,L_Y)$. Since the isomorphism $H^0(X,L_X)\simeq H^0(Y,L_Y)=H^0(Y,L_X+E)$ is given by $s\mapsto (\pi^*s)\otimes s_E$, this shows that $(\X,\L_X)$ and $(\Y,\L_Y)$ induce the same filtration through this isomorphism of section rings.\end{proof}

We may now prove Theorem \ref{thm:pushforwards}.

\begin{proof}[Proof of Theorem \ref{thm:pushforwards}.]
First assume $(Y,L_Y)$ is uniformly K-stable. Let $(\X,\L_X)$ be a test configuration for $(X,L_X)$, which without loss of generality we may assume to be smooth. We can construct a degeneration $\Y$ of $Y$ satisfying the hypotheses of Proposition \ref{prop:pullbacks} through the same procedure as in the proof of Theorem \ref{thm:pullbackinvariance}. Then, Proposition \ref{prop:pullbacks} immediately shows $(X,L_X)$ is uniformly K-stable.

Conversely, assume $(X,L_X)$ is uniformly K-stable. Given a test configuration $(\Y,\L_Y)$ for $(Y,L_Y)$ with associated divisorial filtration of the form $\cF_{(E_i,t_i)}$, we can as in the proof of Theorem \ref{thm:pullbackinvariance} find a degeneration $\X$ of $X$ such that the $E_i$ are realised as components $E_{i,X}$ of $\X_0$ and with a morphism $\pi:\cY\to \cX$. We likewise define a test configuration $(\X,\L_X)$ of $(X,L_X)$ by
\begin{equation}\label{eq:tc2} 
\L_X = \sum_i t_i E_{i,X}+c\sum_j F_j
\end{equation} 
with $c$ sufficiently large, so that the filtration induced by $(\X,\L_X)$ on $R(X,L_X)\simeq R(Y,L_X)$ coincides with the filtration $\cF_{(E_i,t_i)}$ induced by $(\Y,\L_Y)$. Those filtrations are then equal to the filtration induced by the pullback test configuration $(\Y,\L_Y'):=(\Y,\pi^*\L_X + E')$ as in Proposition \ref{prop:pullbacks}. We then conclude using Proposition \ref{prop:pullbacks} and Lemma \ref{lem:comparison} as in the proof of Theorem \ref{thm:pullbackinvariance} that $(Y,L_Y)$ is uniformly K-stable.
\end{proof}

\begin{remark} We briefly outline another proof that, for test configurations for $(X,L_X)$ and $(Y,L_Y)$, the Monge--Amp\`ere measures agree provided the filtrations agree, and that the derivatives of the Monge--Amp\`ere energies agree on the big cone, leaving full details to the interested reader. Consider the Fujita approximations $(X_m,L_{X,m})$ and $(Y_m,L_{Y,m})$ of $(X,L_X)$ and $(Y,L_Y)$, so that the pullback $\pi:Y \to X$ produces a morphism $\pi_m: Y_m \to X_m$ such that $\pi_m^*L_{X,m} = L_{Y_m}$, where $L_{X,m}$ and $L_{Y_m}$ are semiample, but not in general ample. Define $(Z_m,L_{Z,m})$ by taking Proj of the section rings of $(X_m,L_{X,m})$, or equivalently  $(Y_m,L_{Y,m})$, and note that $Y_m$ and $X_m$ have morphisms to $Z_m$ under which the ample line bundle $L_{Z,m}$ pulls back to $L_{Y,m}$ and $L_{X,m}$ respectively.

We recall the construction of Theorem \ref{thm:bchen}, which associates a sequence of test configurations $(\X_m,\L_{X,m})$ to $(X_m,L_{X,m})$, and likewise for $(Y_m,L_{Y,m})$. The test configurations satisfy the condition that $\L_{X,m}$ is relatively semiample, with the test configuration being produced as associated to the filtration on $H^0(X_m,mL_{X,m}) \cong H^0(X,mL_X)$ induced by $(\X,\L_X)$. In particular, the filtrations induced on the (isomorphic) section rings of $(X_m,L_{X,m})$ and $(Y_m,L_{Y,m})$ agree, since the filtrations induced by $(\X,\L_X)$ on $H^0(X,mL_X)$ and by $(\Y,\L_Y)$ on $H^0(Y,mL_Y)$ agree by construction. 

The relative Proj of $(\X_m,\L_{X,m})$ and by $(\Y_m,\L_{Y_m})$ therefore agree, and produce $(\scZ_m, \L_{Z, m})$ which is a test configuration for $(Z_m,L_{Z,m})$, with  $L_{Z,m}$ ample and $\L_{Z, m}$ relatively ample. The Monge--Amp\`ere measures of   $(\Y_m,\L_{Y_m})$ and $( Z_m, \L_{Z, m})$ agree, since they agree with the Monge--Amp\`ere measure of $(\scZ_m, \L_{Z, m})$, by the projection formula. Indeed, $\L_{Z,m}$ pulls back to $\L_{Y,m}$ and $\L_{X,m}$ respectively on $\Y_m$ and $\X_m$, so the projection formula gives equality of the associated  Monge--Amp\`ere measures.

We  claim that the Monge--Amp\`ere measures of  $(\X_m,\L_{X,m})$ suitably converge to  $\MA(\X,\L_{X})$. Consider a component $F_i$ of $\X_0$. Then $$\lim_{m\to\infty} ( \L_{X,m}^n)\cdot F_i = \langle \L_X^n\rangle \cdot F_i ,$$ where on the left hand side we mean the proper transform of $F_i$ in $\X_m$, by Fujita approximation. On the other hand, $F_i$ also corresponds to a component of $\Y_0$ by taking the proper transform, and with respect to this component we also have $$\lim_{m\to\infty} ( \L_{Y,m}^n)\cdot F_i = \langle \L_Y^n\rangle \cdot F_i.$$ Since $ ( \L_{Y,m}^n)\cdot F_i  =  ( \L_{X,m}^n)\cdot F_i$, it follows that for all components of $\X_0$ and $\Y_0$, its  coefficients in $\MA(\Y,\L_Y)$ and $\MA(\X,\L_X)$ are equal. A similar argument implies that the derivatives of the Monge--Amp\`ere energies also agree, by proving equality with the analogous derivative for $(\scZ_m, \L_{Z, m})$.
\end{remark}
	
	\subsection{$b$-divisors}\label{sec:b-div} A trivial consequence of these results is that uniform K-stability may be defined at the level of $b$-divisors. We briefly explain the relevant theory, which may be found for example in \cite{bfj:teissier}. The \emph{Riemann--Zariski space} of $X$ is the projective limit $$\underline X:=\varprojlim_{\{\pi:Y\to X\}} Y$$ over all birational morphisms $\pi: Y \to X$ with $Y$ smooth. The Riemann--Zariski space is  an invariant of the birational isomorphism class of $X$, and admits two classes of divisors: \emph{Cartier $b$-divisors} and \emph{Weil $b$-divisors}. The former are defined as an element in $N^1(Y)$ on birational model $Y$ of $X$; this induces through pullback a class in $N^1$ of any birational model of $X$ with a morphism to $Y$, and such models are cofinal in the projective system. Weil $b$-divisors are defined as classes in $N^1$ of each member of (a cofinal subset of) the projective system, compatible with the natural pushforward operation; in other words, they are elements in the projective limit $$N^1(\underline X):=\varprojlim_{\{\pi:Y\to X\}} N^1(Y).$$ The image of a Weil $b$-divisor by the projection morphism $N^1(\underline X)\to N^1(Y)$ on a model $Y$ is called its \emph{realisation} on $Y$. Since Cartier $b$-divisors naturally form a subset $CN^1(\underline X)\subset N^1(\underline X)$, this also defines the realisation of a Cartier $b$-divisor.

\begin{definition}
We say that a Cartier $b$-divisor $\underline L$ on $\underline X$ is \emph{uniformly K-stable} if its realisation $(Y,L_Y)$ on each $Y$ is uniformly K-stable.
\end{definition}

It follows from Theorem \ref{thm:pullbackinvariance} that this is well-defined, and equivalent to any realisation being uniformly K-stable. 

For a Weil $b$-divisor $\underline L$ on $\underline X$, we restrict to the subsystem $\underline X^L$ of $\underline X$ such that for each $\pi_Y: (Y',L_{Y'}) \to (Y,L_Y)$, we have $L_{Y'} = \pi_Y^*L_Y +E$ for $E$ effective and $\pi_Y$-exceptional.

\begin{definition}
We say that a Weil $b$-divisor $\underline L$ on $\underline X^L$ is \emph{uniformly K-stable} if its realisation $(Y,L_Y)$ on each $Y$ is uniformly K-stable.
\end{definition}

Theorem \ref{thm:pushforwards} similarly implies that this is well-defined, and equivalent to any realisation being uniformly K-stable.

\bibliographystyle{alpha}
\bibliography{bib}

\end{document}